\documentclass[11pt,reqno]{amsart}
\usepackage{amsmath,amssymb,graphicx,mathrsfs,amsopn,times,stmaryrd,amsthm,color,bm}
\usepackage[colorlinks=true,allcolors=blue]{hyperref}
\usepackage{geometry}
\numberwithin{equation}{section}
\newtheorem{thm}{Theorem}[section]
\newtheorem{prop}[thm]{Proposition}
\newtheorem{lem}[thm]{Lemma}
\newtheorem{cor}[thm]{Corollary}

\theoremstyle{definition} 
\newtheorem{eg}[thm]{Example}

\theoremstyle{remark}
\newtheorem{rem}[thm]{Remark}

\newcommand{\beq}{\begin{equation}}
\newcommand{\eeq}{\end{equation}}
\newcommand{\be}{\begin{equation*}}
\newcommand{\ee}{\end{equation*}}
\newcommand{\bs}{\boldsymbol}

\newcommand{\C}{\mathbb{C}}
\newcommand{\bK}{\mathbb{K}}
\newcommand{\Z}{\mathbb{Z}}
\newcommand{\bP}{\mathbb{P}}

\newcommand{\mc}{\mathcal}
\newcommand{\cD}{\mathcal{D}}

\newcommand{\cR}{\mathcal{R}}
\newcommand{\cT}{\mathcal{T}}

\newcommand{\gl}{\mathfrak{gl}}
\newcommand{\h}{\mathfrak{h}}

\newcommand{\fkS}{\mathfrak{S}}

\newcommand{\Wr}{\mathrm{Wr}}
\newcommand{\rY}{\mathrm{Y}}
\newcommand{\End}{\mathrm{End}}
\newcommand{\ord}{\mathrm{ord}}
\newcommand{\id}{{\mathrm{id}}}   

\newcommand{\sF}{\mathscr{F}}

\newcommand{\pa}{\partial}
\newcommand{\tl}{\tilde}
\newcommand{\gge}{\geqslant}
\newcommand{\lle}{\leqslant}
\newcommand{\la}{\lambda}

\newcommand{\bla}{\bm\lambda}

\newcommand{\glMN}{\mathfrak{gl}_{m|n}}
\newcommand{\UglMN}{\mathrm{U}(\mathfrak{gl}_{m|n})}
\newcommand{\YglMN}{\mathrm{Y}(\mathfrak{gl}_{m|n})}
\newcommand{\bmx}{\begin{pmatrix}}    
\newcommand{\emx}{\end{pmatrix}}   
\newcommand{\wt}{\widetilde}    
\newcommand{\ka}{\kappa}
\newcommand{\bka}{\bm{\kappa}}
\newcommand{\qedd}{\tag*{$\square$}}

\begin{document}
\pagestyle{myheadings}
\setcounter{page}{1}

\title[Bethe ansatz equation and rational difference operators]{Solutions of $\glMN$ XXX Bethe ansatz equation\\ and rational difference operators}

\author{Chenliang Huang, Kang Lu, and Evgeny Mukhin}
\address{C.H.: Department of Mathematical Sciences,
Indiana University-Purdue University\newline
\strut\kern\parindent Indianapolis, 402 N.Blackford St., LD 270,
Indianapolis, IN 46202, USA}\email{ch30@iupui.edu}
\address{K.L.: Department of Mathematical Sciences,
Indiana University-Purdue University\newline
\strut\kern\parindent Indianapolis, 402 N.Blackford St., LD 270,
Indianapolis, IN 46202, USA}\email{lukang@iupui.edu}
\address{E.M.: Department of Mathematical Sciences,
Indiana University-Purdue University\newline
\strut\kern\parindent Indianapolis, 402 N.Blackford St., LD 270,
Indianapolis, IN 46202, USA}\email{emukhin@iupui.edu}

\begin{abstract} We study solutions of the Bethe ansatz equations of the non-homogeneous periodic XXX model associated to super Yangian $\YglMN$.
To a solution we associate a rational difference operator $\cD$ and a superspace of rational functions $W$. We show that the set of complete factorizations of $\cD$ is in canonical bijection with the variety of superflags in $W$ and that each generic superflag defines a solution of the Bethe ansatz equation. We also give the analogous statements for the quasi-periodic supersymmetric spin chains.

\medskip

\noindent
{\bf Keywords:} supersymmetric spin chains, Bethe ansatz, difference operators.
\end{abstract}

\maketitle
\thispagestyle{empty}
\section{Introduction}	
The supersymmetric spin chains were introduced back to \cite{K} in 1980s. 
There is a considerable renewed interest to those models, see \cite{BR,BR2,KSZ,PRS,TZZ}. However,  many results available for the even spin chains are still unknown for the supersymmetric case. In this paper we are able to fill up a few gaps.

We use the method of populations of solutions of the Bethe ansatz equations. It was pioneered in \cite{MV1} in the case of the Gaudin model and then extended to the XXX models constructed from the Yangian associated to $\gl_n$, see \cite{MV03,MV2,MTV2}. We are helped by the recent work on the populations of the supersymmetric Gaudin model \cite{HMVY18}.

\medskip

Let us describe our findings in more detail. In this paper we restrict ourselves to tensor products of evaluation polynomial $\gl_{m|n}$-modules. Moreover, we assume that the evaluation parameters are generic, meaning they are distinct modulo $h\Z$ where $h$ is the shift in the super Yangian relations. Note that such tensor products are irreducible $\YglMN$-modules. We also assume that at least one of the participating $\gl_{m|n}$-modules is typical.

The crucial observation is the reproduction procedure which given a solution of the Bethe ansatz equation and a simple root of $\glMN$, produces another solution, see Theorem \ref{thm repro pro glmn}. 

The reproduction procedure along an even root is given in \cite{MV03}. An even component of a solution of the Bethe ansatz equation gives a polynomial solution of a second order difference equation. The reproduction procedure amounts to trading this solution to any other polynomial solution of the difference equation, see \eqref{bosonic rp}. We call it the \emph{bosonic reproduction procedure}. 

The reproduction procedure along an odd root is different. In fact, an odd component of a solution of the Bethe ansatz equation corresponds to a polynomial which divides some other polynomial, see \eqref{fermionic rp}. The reproduction procedure changes the divisor to the quotient polynomial with an appropriate shift. We call it the \emph{fermionic reproduction procedure}. The fermionic reproduction procedure looks similar to a mutation in a cluster algebra.

Then the population is the set of all solutions obtained from one solution by recursive application of the reproduction procedure. 

\medskip

Given a solution of the Bethe ansatz equation, we define a rational difference operator of the form $\cD=\cD_{\bar 0} \cD_{\bar 1}^{-1}$, where $\cD_{\bar 0},$ $\cD_{\bar 1}$ are linear difference operators of orders $m$ and $n$ with rational coefficients, respectively, see \eqref{eq rational diff oper glmn}. The operator $\cD$ is invariant under reproduction procedures and therefore it is defined for the population, see Theorem \ref{thm diffoper inv}. The idea of considering such an operator is found in \cite{HMVY18} in the case of the Gaudin model. Such an operator in the case of tensor products of vector representations also appears in \cite{T} in relation to the study of T-systems and analytic Bethe ansatz.

Kernels $V=\ker \cD_{\bar 0}$,  $U=\ker \cD_{\bar 1}$ are spaces of rational functions of dimensions $m$ and $n$. Under our assumption, that at least one of the representations is typical, we can show  $V\cap U=0$, see Lemma \ref{lem empty cap}. We consider superspace $W=V\oplus U$. Then we show that there are natural bijections between three objects: elements of the population of the solutions of the Bethe ansatz equation, superflags in $W$, and complete factorizations of $\cD$ into products of linear difference operators and their inverses, see Theorem \ref{thm bijection 3 objs}.

Note that the Bethe ansatz equations depend on the choice of the Borel subalgebra in $\gl_{m|n}$. The fermionic reproductions change this choice. In general, the Borel subalgebra is determined from the parity of the superflag or, equivalently, from the positions of the inverse linear difference operators in a complete factorization of $\cD$.

\medskip

Thus the solutions of the Bethe ansatz equations correspond to superspaces of rational functions. It is natural to expect that all joint eigenvectors of XXX Hamiltonians correspond to such spaces and that there is a natural correspondence between the eigenvectors of the transfer matrix and points of an appropriate Grassmannian. However, the precise formulation of this correspondence is not established even in the even case, see \cite{MTV2}. 

\medskip

We give a few details in the quasi-periodic case as well, see Section \ref{sec quasi}. In this case we also have concepts of reproduction procedure, 
the population, and the rational difference operator. Then the elements in the population are in a natural bijection with the permutations of the distinguished flags in the space of functions of the form $f(x)=e^{zx}r(x)$, where $r(x)\in\C(x)$ is a rational function and $z\in\C$, see Theorem \ref{thm bijection weyl}. A similar picture in the even case is described in \cite{MV2}.

\medskip

The paper is constructed as follows. In Section \ref{sec rational diff oper study} we study rational difference operators and their complete factorizations. We then recall the XXX model associated to $\gl_{m|n}$ and the corresponding Bethe ansatz equations in Section \ref{sec XXX}. In Section \ref{sec rep pro gl2 gl11} we recall the reproduction procedure for $\gl_2$ and define its analog for $\gl_{1|1}$. In Section \ref{sec glmn rep pro} we define reproduction procedures for $\gl_{m|n}$, a population, and a rational difference operator associated to a population. In Section \ref{sec pop flag variety} we give the bijections between the superflag variety, the set of complete factorizations, and a population. We conclude our paper by generalizing our results to the quasi-periodic XXX model in Section \ref{sec quasi}. Appendix \ref{sec gl11 eg conj} is devoted to the basics of Bethe ansatz in the case of $\mathrm{Y}(\gl_{1|1})$.

\medskip 

{\bf Acknowledgments.} This work was partially supported by a grant from the Simons Foundation \#353831.

\section{Rational difference operators and their factorizations}
\label{sec rational diff oper study}
We study properties of ratios of difference operators, following the treatment of ratios of differential operators in \cite{CDSK}. We also describe the relation between the complete factorizations and the superflag varieties.

\subsection{Parity sequences}
We use the notation of \cite[Section 2]{HMVY18}. We recall some of them.

Denote by $S_{m|n}$ the set of all sequences $\bm s=(s_{1},s_2,\dots,s_{m+n})$ where $s_i\in\{\pm1\}$ and $1$ occurs exactly $m$ times. Elements of $S_{m|n}$ are called \emph{parity sequences}. The parity sequence of the form $\bm s_0=(1,\dots,1,-1,\dots,-1)$ is the \emph{standard parity sequence}. 

A parity sequence $\bm s$ corresponds to a permutation $\sigma_{\bm s}$ of the permutation group $\fkS_{m+n}$ of $m+n$ elements as follows:
\[
\sigma_{\bm s}(i)=\begin{cases}
\#\{j~|~j\lle i,~s_j=1\},&\text{ if } s_i=1,\\
m+\#\{j~|~j\lle i,~s_j=-1\},&\text{ if } s_i=-1.
\end{cases}
\]
For a parity sequence $\bm s\in S_{m|n}$, we define
\[
\bm s_i^+=\#\{j~|~j>i,~s_j=1\},\quad \bm s_i^-=\#\{j~|~j<i,~s_j=-1\},\quad i=1,\dots,m+n.
\]
The permutation $\sigma_{\bm s}$ is related to $\bm s_i^{\pm}$ by
\be
\bs s_i^+=\begin{cases} m-\sigma_{\bs s}(i), & \mbox{if } s_i=1,\\  \sigma_{\bs s}(i)-i,  & \mbox{if } s_i=-1, \end{cases}
\qquad 
\bs s_i^-=\begin{cases}i-\sigma_{\bs s}(i), & \mbox{if } s_i=1,\\  \sigma_{\bs s}(i)-m-1,  & \mbox{if } s_i=-1. \end{cases}
\ee
If the parity sequence is dropped from the notation, it means we consider the standard parity sequence.

\subsection{Rational difference operators}
Fix a non-zero number $h\in \C^\times$. Let $\bK$ be the field of complex valued rational functions $\bK=\C(x)$, with an automorphism $\tau:\bK\to \bK$, $(\tau f)(x)\mapsto f(x-h)$. 

Consider the algebra $\bK[\tau]$ of {\it difference operators} where the shift operator $\tau$ 
satisfies $$\tau \cdot f=f(x-h)\cdot \tau$$ for all $f\in\bK$. By definition, an element $\cD\in\bK[\tau]$ has the form 
\beq\label{difference oper}
\cD=\sum_{j=0}^{r}a_{j}\tau^{j},\quad a_j\in\bK,\quad r\in\Z_{\gge 0}.
\eeq
The difference operator $\cD$ has {\it order} $r$, $\ord~\cD=r$, if $a_r\neq0$. One says that $\cD$ is {\it monic} if $a_r=1$. We call $a_0$ the \emph{constant term} of $\cD$.

Let $\cD\in\bK[\tau]$ be a difference operator of order $r$ as in \eqref{difference oper}. We say a difference operator $\cD$ of order $r$ is \emph{completely factorable over $\bK$} if there exist $f_i\in\bK$, $i=1,\dots,r$, such that $\cD=a_r\,d_1\dots d_r$, where $d_i=\tau-f_i$. We focus on completely factorable difference operators with non-zero constant terms $a_0$. In this case, we consider factorizations of the form $\cD=a_0d_1\cdots d_r$, where $d_i=1-\tl f_i\tau$, $\tl f_i\in\bK$, $i=1,\dots,r$. 

Let $\ker \cD=\{u\in\bK\;|\;\cD u=0\}$ be the kernel of $\cD$. It is clear that if $\dim\left(\ker \cD\right)=\ord\,\cD$, then $\cD$ is completely factorable over $\bK$.

\medskip

Let $\bK(\tau)$ be the division ring generated by $\bK[\tau]$. The division ring $\bK(\tau)$ is called the \emph{ring of rational difference operators}. Elements in $\bK(\tau)$ are called \emph{rational difference operators}.

A \emph{fractional factorization} of a rational difference operator $\cR$ is the equality $\cR=\cD_{\bar 0}\cD_{\bar 1}^{-1}$, where $\cD_{\bar 0},\cD_{\bar 1}\in\bK[\tau]$. A fractional factorization $\cR=\cD_{\bar 0}\cD_{\bar 1}^{-1}$ is called \emph{minimal} if $\cD_{\bar 1}$ is monic and has the minimal possible order. 

\begin{prop}\label{prop rdp}
Any  rational difference operator $\cR\in\bK(\tau)$ has the following properties.
\begin{enumerate} 
\item There exists a unique minimal fractional factorization of $\cR$.
\item Let $\cR=\cD_{\bar 0}\cD_{\bar 1}^{-1}$ be the minimal fractional factorization. If $\cR=\wt{\cD}_{\bar 0}\wt{\cD}_{\bar 1}^{-1}$ is a fractional factorization, then there exists $\cD\in\bK[\tau]$ such that $\widetilde{\cD}_{\bar 0}=\cD_{\bar 0}\cD$ and $\widetilde{\cD}_{\bar 1}=\cD_{\bar 1}\cD$.
\item Let $\cR=\cD_{\bar 0}\cD_{\bar 1}^{-1}$ be a fractional factorization such that $\dim\left(\ker \cD_{\bar 0}\right)=\ord~ \cD_{\bar 0} $ and $\dim\left(\ker \cD_{\bar 1} \right)=\ord~ \cD_{\bar 1} $. Then $\cR=\cD_{\bar 0}\cD_{\bar 1}^{-1}$ is the minimal fractional factorization of $\cR$ if and only if $\ker \cD_{\bar 0}\cap\ker \cD_{\bar 1}=0$.
\end{enumerate}
\end{prop}
\begin{proof}
We have the analogs of \cite[Proposition 2.1, Corollary 2.2, Lemma 3.2]{CDSK} for difference operators. Namely, the algebra $\bK[\tau]$ is right Euclidean, therefore $\bK[\tau]$ satisfies the right Ore condition and every right ideal of $\bK[\tau]$ is principal. This statement is proved similarly as \cite[Proposition 3.4]{CDSK}.
\end{proof}

We call $\cR$ an \emph{$(m|n)$-rational difference operator} if in the minimal fractional factorization 
$\cR=\cD_{\bar 0}\cD_{\bar 1}^{-1}$, $\cD_{\bar 0},\cD_{\bar 1}$ are completely factorable over $\bK$, and $\ord(\cD_{\bar 0})=m$, $\ord(\cD_{\bar 1})=n$, and $\cD_{\bar 0},\cD_{\bar 1}$ have the same non-zero constant term.

Let $\cR$ be an $(m|n)$-rational difference operator. Note that $\cR$ can also be written in the form $\cR=\wt {\cD}_{\bar 1}^{-1} \wt{\cD}_{\bar 0}$, where 
$\wt {\cD}_{\bar 1},\wt {\cD}_{\bar 0}\in\bK[\tau]$, $\ord(\wt{\cD}_{\bar 0})=m$, and $\ord(\wt{\cD}_{\bar 1})=n$. More generally, let $\bs s\in S_{m|n}$ be a parity sequence. Then we call the form $\cR=d_1^{s_1}\dots d_{m+n}^{s_{m+n}}$, where $d_i=1-f_i\tau$, $f_i\in\bK$, $i=1,\dots,m+n$, a \emph{complete factorization with the parity sequence $\bs s$}. Let $\mathfrak F^{\bs s}(\cR)$ be the set of all complete factorizations of $\cR$ with parity sequence $\bs s$ and $\mathfrak F(\cR)=\bigsqcup_{\bm s \in S_{m|n}}\mathfrak F^{\bs s}(\cR) $ the set of all complete factorizations of $\cR$.

Throughout the paper, we use the following useful notation: for any $i\in\Z$ and $f\in\bK$,
\[
f[i]:=\tau^i (f)=f(x-ih).
\]
Define the \emph{discrete logarithmic derivative} of a function $f(x)$ by $\ln'(f)=f/f[1]$. 

Consider two $(1|1)$-rational difference operators
\[
\cR_1=(1-a\,\tau)(1-b\,\tau)^{-1}\quad\text{and}\quad \cR_2=(1-c\,\tau)^{-1}(1-d\,\tau),
\]
where $a,b,c,d\in \bK$, $a\ne b$, and $c\ne d$. 

\begin{lem}\label{eq relations diff}
We have $\cR_1=\cR_2$ if and only if
\be
\begin{cases}
	c=b[1]\ln'(a-b),\\
	d=a[1]\ln'(a-b),\\
\end{cases}
\quad\text{or equivalently}\qquad
\begin{cases}
	a[1]=d/\ln'(c-d),\\
	b[1]=c/\ln'(c-d).\\
\end{cases}\qedd
\ee	
\end{lem}

Let $\cR$ be an $(m|n)$-rational difference operator with a complete factorization $\cR=d_1^{s_1}\cdots d_{m+n}^{s_{m+n}}$, where $d_i=1-f_i\tau$. Suppose $s_i\ne s_{i+1}$ and $d_i\ne d_{i+1}$. Using Lemma \ref{eq relations diff}, one constructs $\tl{d}_i$ and $\tl{d}_{i+1}$ such that $d_i^{s_i}d_{i+1}^{s_{i+1}}=\tl d_{i}^{s_{i+1}}\tl d_{i+1}^{s_i}$. This induces a new complete factorization of $\cR=d_1^{s_1}\cdots \tl d_{i}^{s_{i+1}}\tl d_{i+1}^{s_i}\cdots  d_{m+n}^{s_{m+n}}$ with the new parity sequence $\tl{\bm s}=\bm s^{[i]}=(s_1,\dots,s_{i+1},s_{i},\dots,s_{m+n})$.

Repeating this procedure, we see that there exists a canonical bijection between the sets of complete factorizations with respect to any two parity sequences.

\subsection{Complete factorizations and superflag varieties}
\label{sec 2.3}
Let $W=W_{\bar 0}\oplus W_{\bar 1}$ be a vector superspace with $\dim(W_{\bar 0})=m$ and $\dim(W_{\bar 1})=n$. Consider a \emph{full flag} $\sF$ of $W$, $\sF=\{F_1\subset F_2\subset \dots\subset F_{m+n}=W\}$ such that $\dim(F_i)=i$. A basis $\{w_1,\dots,w_{m+n}\}$ of $W$ \emph{generates the full flag $\sF$} if $F_i$ is spanned by $w_1,\dots,w_i$. A full flag is called a \emph{full superflag} if it is generated by a homogeneous basis. We denote by $\sF(W)$ the set of all full superflags.

To a homogeneous basis $\{w_1,\dots,w_{m+n}\}$ of $W$, we associate the unique parity sequence $\bm s\in S_{m|n}$ such that $s_i=(-1)^{|w_i|}$. We say a full superflag $\sF$ \emph{has parity sequence} $\bm s$ if it is generated by a homogeneous basis whose parity sequence is $\bm s$. We denote by $\sF^{\bm s}(W)$ the set of all full superflags of parity $\bm s$.

Clearly, we have
\[
\sF(W)=	\bigsqcup_{\bm{s}\in S_{m|n}} \sF^{\bm s}(W), \qquad 
\sF^{\bm s}(W)\cong\sF\left(W_{\bar{0}}\right)\times \sF\left(W_{\bar{1}}\right).
\] 

Given a basis $\{v_1,\dots, v_m\}$ of $W_{\bar 0}$, a basis $\{u_1,\dots, u_n\}$ of $W_{\bar 1}$, and a parity sequence $\bm s\in S_{m|n}$, define a homogeneous basis $\{w_1,\dots, w_{m+n}\}$ of $W$ by the rule $w_i=v_{\bm s^+_i+1}$ if $s_i=1$ and 
$w_i=u_{\bm s^-_i+1}$ if $s_i=-1$. Conversely, any homogeneous basis of $W$ gives a basis of $W_{\bar 0}$, a basis of $W_{\bar 1}$, and a parity sequence $\bm s$. We say that the basis $\{w_1,\dots, w_{m+n}\}$ is \emph{associated to $\{v_1,\dots, v_m\}$, $\{u_1,\dots, u_n\}$, and $\bm s$}.

Define the \emph{discrete Wronskian} $\Wr$ (or Casorati determinant) of $g_1,\dots,g_r$ by
$$
\Wr^{\pm}(g_1,\dots,g_r)=\det \left( g_j[\mp(i-1)] \right)_{i,j=1}^r=\det \left( g_j(x\pm(i-1)h) \right)_{i,j=1}^r.
$$
We simply write $\Wr$ for $\Wr^-$.

Let $\cR$ be an $(m|n)$-rational difference operator over $\bK$. Let $\cR=\cD_{\bar 0}\cD_{\bar 1}^{-1}$ be a fractional factorization such that $\mathrm{ord}~\cD_{\bar 1}=n$ and the constant term of $\cD_{\bar 1}$ is $1$. By Proposition \ref{prop rdp}, such a fractional factorization of $\cR$ is unique.

Let $V=W_{\bar 0}=\ker \cD_{\bar 0} $, $U=W_{\bar 1}=\ker \cD_{\bar 1} $, $W=W_{\bar 0}\oplus W_{\bar 1}$.

Given a basis $\{v_1,\dots, v_m\}$ of $V$, a basis $\{u_1,\dots, u_n\}$ of $U$, and a parity sequence $\bm s\in S_{m|n}$, define $d_i=1-f_i\tau$, where
\beq\label{eq wronski coeff}
\begin{split}
&f_i= \ln'  \frac{\Wr(v_1,v_2,\dots,v_{\bm s_i^++1},u_1,u_2,\dots,u_{\bm s_i^-})}{\Wr(v_1,v_2,\dots,v_{\bm s_i^+},u_1,u_2,\dots,u_{\bm s_i^-})[1]}, \qquad {\rm if}\ s_i=1, \\ &f_i= \ln'  \frac{\Wr(v_1,v_2,\dots,v_{\bm s_i^+},u_1,u_2,\dots,u_{\bm s_i^-+1})}{\Wr(v_1,v_2,\dots,v_{\bm s_i^+},u_1,u_2,\dots,u_{\bm s_i^-})[1]},  \qquad {\rm if}\ s_i=-1.
\end{split}
\eeq

Note that if two bases $\{v_1,\dots, v_m\}$, $\{\tl{v}_1,\dots, \tl{v}_m\}$ generate the same full flag of $V$ and two bases $\{u_1,\dots, u_n\}$, $\{\tl{u}_1,\dots, \tl{u}_n\}$ generate the same full flag of $U$, then the coefficients $f_i$ computed from $v_j,u_j$ and from $\tl{v}_j,\tl{u}_j$ are the same.

\begin{prop}\label{prop flag factor}
We have a complete factorization of $\cR$ with parity $\bm s$: $\cR=d_1^{s_1}\cdots d_{m+n}^{s_{m+n}}$.
\end{prop}
\begin{proof}
The statement for the case of $\bm s=\bm s_0$ follows from \cite{MV03}.

Let $\bm s$ and $\tl{\bm s}$ be two parity sequences which differ only in positions $i, i+1$. Explicitly, $s_j=\tl {s}_j$ for $j\neq i, i+1$ and $s_i=-s_{i+1}=-\tl{s}_i=\tl {s}_{i+1}$.
It is clear that $d_j=\tl {d}_j$ for $j\neq i, i+1$. In addition, the equality $d_i^{s_i}d_{i+1}^{s_{i+1}}=\tl {d}_i^{\ \tl{s}_i}\tl{d}_{i+1}^{\ \tl{s}_{i+1}}$ follows from the discrete Wronskian identity, see \cite[Lemma 9.5]{MV03},
\begin{align*}
&\Wr\big(\Wr(v_1,v_2,\dots,v_{\bm s_i^++1},u_1,u_2,\dots,u_{\bm s_i^-}),\Wr(v_1,v_2,\dots,v_{\bm s_i^+},u_1,u_2,\dots,u_{\bm s_i^-+1})\big) \\
=\,&
\Wr(v_1,v_2,\dots,v_{\bm s_i^++1},u_1,u_2,\dots,u_{\bm s_i^-+1})\Wr(v_1,v_2,\dots,v_{\bm s_i^+},u_1,u_2,\dots,u_{\bm s_i^-})[1].\qedhere
\end{align*}
\end{proof}

By Proposition \ref{prop flag factor}, we have maps $\varpi: \sF(W) \to \mathfrak F(\cR)$ and $\varpi^{\bm s}: \sF^{\bm s}(W) \to \mathfrak F^{\bm s}(\cR)$. 
\begin{cor}
The maps $\varpi$ and $\varpi^{\bm s}$ are bijections.\qed
\end{cor}
Thus the set of complete factorizations of $\cR$ is canonically identified with the variety of full superflags of $W$. 

\section{XXX model}\label{sec XXX}
In this section we recall the definition of the super Yangian $\YglMN$ and some facts about the XXX model associated with $\YglMN$. Our main source is \cite{BR}.

\subsection{Super Yangian $\YglMN$ and transfer matrix}
Let $\C^{m|n}$ be the complex vector superspace with $\dim(\C^{m|n}_{\bar 0})=m$ and $\dim(\C^{m|n}_{\bar 1})=n$. We choose a homogeneous basis $e_1,\dots,e_{m+n}$ of $\C^{m|n}$ such that $|e_i|=0$ for $1\lle i\lle m$ and $|e_j|=1$ for $m+1\lle j\lle m+n$. Denote by $E_{ij}\in\End(\C^{m|n})$ the linear operator of parity $|i|+|j|$ such that $E_{ij}e_k=\delta_{jk}e_i$ for $1\lle i,j,k\lle m+n$.

The \emph{super Yangian} $\YglMN$ is a unital associative algebra with generators $\mc L_{ij}^{(k)}$ of parity $|i|+|j|$, $i,j=1,\dots,m+n$, $k\in \Z_{>0}$.

Consider the generating series
\[
\mc L_{ij}(x)=\sum_{k=0}^\infty \mc L_{ij}^{(k)}x^{-k}, \qquad \mc L_{ij}^{(0)}=\delta_{ij},
\]
and combine the series into a linear operator $\mc L(x)=\sum_{i,j=1}^{m+n}E_{ij}\otimes \mc L_{ij}(x)\in \End(\C^{m|n})\otimes \YglMN[[x^{-1}]]$. The defining relations of $\YglMN$ are given by
\beq\label{RTT}
R^{(12)}(x_1-x_2)\mc L^{(13)}(x_1)\mc L^{(23)}(x_2)=\mc L^{(23)}(x_2)\mc L^{(13)}(x_1)R^{(12)}(x_1-x_2),
\eeq
where $R(x)\in \End(\C^{m|n}\otimes \C^{m|n})$ is the super R-matrix defined by
\[
x\, R(x)=x\,\id +h\sum_{i,j=1}^{m+n}(-1)^{|j|}E_{ij}\otimes E_{ji}.
\]
\begin{rem}
Note that, for any non-zero $z\in\C^{\times}$, the map $\mc L_{ij}(x)\mapsto \mc L_{ij}(x/z)$ induces an automorphism of $\YglMN$, therefore the super Yangians $\YglMN$ defined by different non-zero $h$ are actually isomorphic. In particular, we can always rescale $h$ to $1$.\qed
\end{rem}

The R-matrix $R(x)$ satisfies the graded Yang-Baxter equation,
\[
R^{(12)}(x_1-x_2)R^{(13)}(x_1)R^{(23)}(x_2)=R^{(23)}(x_2)R^{(13)}(x_1)R^{(12)}(x_1-x_2).
\]

The super commutator relations obtained from \eqref{RTT} are explicitly given by
\beq\label{com relations}
\begin{split}
(x_1-x_2)[\mc L_{ij}(x_1),\mc L_{k\ell}(x_2)]=&\ (-1)^{|i||k|+|\ell||i|+|\ell||k|}h\big(\mc L_{kj}(x_2)\mc L_{i\ell}(x_1)-\mc L_{kj}(x_1)\mc L_{i\ell}(x_2)\big)\\
=&\ (-1)^{|i||j|+|\ell||i|+|\ell||j|}h\big(\mc L_{i\ell}(x_1)\mc L_{kj}(x_2)-\mc L_{i\ell}(x_2)\mc L_{kj}(x_1)\big).
\end{split}
\eeq
In particular, one has
\beq\label{zero mode com relations}
[\mc L_{ij}^{(1)},\mc L_{k\ell}(x)]=(-1)^{|i||k|+|\ell||i|+|\ell||k|}h\big(\delta_{i\ell}\mc L_{kj}(x)-\delta_{kj}\mc L_{i\ell}(x)\big).
\eeq

The super Yangian $\YglMN$ is a Hopf algebra with the coproduct
\[
\Delta: \mc L_{ij}(x)\mapsto \sum_{k=1}^{m+n} (-1)^{(|k|+|i|)(|k|+|j|)}\mc L_{ik}(x)\otimes \mc L_{kj}(x),\qquad i,j=1,\dots,m+n.
\]
The super Yangian $\YglMN$ contains the algebra $\UglMN$ as a Hopf subalgebra. The embedding is given by the map $e_{ij}\mapsto (-1)^{|i|}\mc L_{ji}^{(1)}/h$ for $1\lle i,j\lle m+n$. We identify $\UglMN$ with the image of this map.

The \emph{transfer matrix} $\mc T(x)$ is defined as the supertrace of $\mc L(x)$,
\[
\mc T(x)=\mathrm{str}(\mc L(x))=\sum_{i=1}^{m+n}(-1)^{|i|}\mc L_{ii}(x).
\]
It is known that the transfer matrices commute, $[\mc T(x_1),\mc T(x_2)]=0$. Moreover, the transfer matrix $\cT(x)$ commutes with the subalgebra $\UglMN$.

Since the transfer matrices commute, the transfer matrix can be considered as a generating function of integrals of motion of an integrable system.

\medskip

For any given complex number $z\in \C$, there is an automorphism
\[
\zeta_z:\YglMN\to \YglMN,\qquad \mc L_{ij}(x)\to \mc L_{ij}(x-z),
\]
where $(x-z)^{-1}$ is expanded as a power series in $x^{-1}$. The \emph{evaluation homomorphism} $\mathrm{ev}:\YglMN\to \UglMN$ is defined by the rule: $$\mc L_{ji}^{(a)}\mapsto (-1)^{|i|}\delta_{1a}he_{ij},$$ for $a\in\Z_{>0}$.

For any $\glMN$-module $V$ denote by $V(z)$ the $\YglMN$-module obtained by pulling back of $V$ through the homomorphism $\mathrm{ev}\circ\zeta_z$. The module $V(z)$ is called the \emph{evaluation module with the evaluation point} $z$.

Let $V$ be a $\YglMN$-module. Given a parity sequence $\bs s\in S_{m|n}$, a non-zero vector $v\in V$ is called an \emph{$\bs s$-singular vector} if
\[
\mc L_{ii}^{\bm s}(x)v=\Lambda_i(x)v,\qquad \mc L_{ij}^{\bm s}(x)v=0,\quad i>j,
\]
where $\Lambda_i(x)\in \C[[x^{-1}]]$ and $\mc L_{a,b}^{\bm s}(x)=\mc L_{\sigma_{\bs s}(a),\sigma_{\bs s}(b)}(x)$.

\begin{eg}
Let $L_{\lambda}$ be an irreducible polynomial $\glMN$-module of highest weight $\la$ with highest weight vector $v_\lambda$. Let $z$ be a complex number. Then the $\glMN$ $\bm s$-singular vector $v_\la^{\bm s}\in L_{\lambda}(z)$ is a $\YglMN$ $\bm s$-singular vector. Moreover, we have
\[
\mc L_{ii}^{\bm s}(x)v_\la^{\bm s}=\Big(1+\frac{s_i\,\la^{\bm s}(e_{ii}^{\bm s})h}{x-z}\Big)v_\la^{\bm s}=\frac{x-z+s_i\,\la^{\bm s}(e_{ii}^{\bm s})h}{x-z}v_\la^{\bm s},\quad i=1,2,\dots,m+n.\qedd
\]
\end{eg}

\subsection{Bethe ansatz equation}\label{sec super XXX bae}
We fix a parity sequence $\bs s\in S_{m|n}$, a sequence $\bs \lambda=(\lambda^{(1)},\dots,\lambda^{(p)})$ of polynomial $\gl_{m|n}$ weights, and a sequence $\bs z=(z_1,\dots,z_p)$ of complex numbers. We call $(\lambda^{(k)})^{\bm{s}}$, see \cite[Section 2]{HMVY18},  the \emph{weight at point $z_k$ with respect to $\bm{s}$}. We simply write $\la_i^{(k,\bm s)}$ for $(\lambda^{(k)})^{\bm{s}}(e_{ii}^{\bm s})$.

Let $\bs l=(l_1,\dots,l_{m+n-1})$ be a sequence of non-negative integers. Define $l=\sum_{i=1}^{m+n-1}l_i$. Let $\bm{t}=(t_{1}^{(1)},\dots,t_{l_1}^{(1)};\dots;t_{1}^{(m+n-1)},\dots,t_{l_{m+n-1}}^{(m+n-1)})$ be a collection of variables.  We say that $t_j^{(i)}$ has \emph{color $i$}. Define the \emph{$\glMN$ weight at $\infty$ with respect to $\bm{s}$, $\bm{\lambda}$, and $\bm{l}$} by 
\[
\lambda^{(\bs s,\infty)}=\sum_{k=1}^{p}(\lambda^{(k)})^{\bm{s}}-\sum_{i=1}^{m+n-1}l_i \alpha^{\bs s}_{i} .
\]

The \emph{Bethe ansatz equation} (BAE) \emph{associated to $\bs s$, $\bs z$, $\bm\lambda$, and $\bm{l}$}, is a system of algebraic equations in variables $\bs t$:
\begin{equation}\label{eq BAE XXX}
\resizebox{0.93\hsize}{!}{%
$\displaystyle\prod_{k=1}^p\frac{t_j^{(i)}-z_k+s_i\la_i^{(k,\bm s)}h}{t_j^{(i)}-z_k+s_{i+1}\la_{i+1}^{(k,\bm s)}h}\prod_{r=1}^{l_{i-1}}\frac{t_j^{(i)}-t_r^{(i-1)}+s_ih}{t_j^{(i)}-t_r^{(i-1)}}\prod_{\substack{r=1 \\ r\ne j}}^{l_{i}}\frac{t_j^{(i)}-t_r^{(i)}-s_ih}{t_j^{(i)}-t_r^{(i)}+s_{i+1}h}
\prod_{r=1}^{l_{i+1}}\frac{t_j^{(i)}-t_r^{(i+1)}}{t_j^{(i)}-t_r^{(i+1)}-s_{i+1}h}=1,$}
\end{equation}
where $i=1,\dots,m+n-1$, $j=1,\dots,l_i$. We call the single equation \eqref{eq BAE XXX} the {\it BAE for $\bs t$ related to $t_j^{(i)}$}.

We allow the following cancellations in the BAE,
\beq\label{eq cancel}
\frac{t_j^{(i)}-z_k+s_i\la_i^{(k,\bm s)}h}{t_j^{(i)}-z_k+s_{i+1}\la_{i+1}^{(k,\bm s)}h}=1, \text{ if } s_i\la_i^{(k,\bm s)}=s_{i+1}\la_{i+1}^{(k,\bm s)};\quad  \frac{t_j^{(i)}-t_r^{(i)}-s_ih}{t_j^{(i)}-t_r^{(i)}+s_{i+1}h}=1, \text{ if } s_i=-s_{i+1}.
\eeq
After these cancellations, we consider only the solutions that do not make the remaining denominators in \eqref{eq BAE XXX} vanish.

In addition, we impose the following condition. Suppose $(\alpha_i^{\bm s},\alpha_i^{\bm s})=0$ for some $i$. Consider the BAE for $\bm t$ related to $t_j^{(i)}$ with all $t_b^{(a)}$ fixed, where $a\ne i$ and $1\lle b\lle l_a$, this equation does not depend on $j$. Let $t_0^{(i)}$ be a solution of this equation with multiplicity $r$. Then we require that the number of $j$ such that $t_j^{(i)}=t_0^{(i)}$ is at most $r$, c.f. Lemma \ref{eq gl11 XXX BAE}, Theorem \ref{thm repro pro glmn}.

The group $\mathfrak{S}_{\bs l}=\mathfrak{S}_{l_1}\times\dots\times\mathfrak{S}_{l_{m+n-1}}$ acts on $\bs t$ by permuting the variables of the same color. 

We do not distinguish between solutions of the BAE in the same $\mathfrak{S}_{\bs l}$-orbit.

\begin{rem}\label{rem BAE limit}
Note that in the quasiclassical limit $h\to 0$, system \eqref{eq BAE XXX} becomes system (4.2) of
\cite{MVY}, which is the Bethe ansatz equation of Gaudin model associated to $\gl_{m|n}$.\qed
\end{rem}

\subsection{Bethe vector}
Let $\bs \lambda=(\lambda^{(1)},\dots,\lambda^{(p)})$ be a sequence of polynomial $\gl_{m|n}$ weights. Let $v_k^{\bs s}=v^{\bs s}_{(\lambda^{(k)})^{\bm{s}}}$ be an $\bs s$-singular vector in the irreducible $\gl_{m|n}$-module $L_{\lambda^{(k)}}$. Consider the tensor product of evaluation modules $L(\bm\lambda,\bm z)=\bigotimes_{k=1}^p L_{\lambda^{(k)}}(z_k)$. We also denote by $L(\bm\lambda)$ the corresponding $\glMN$-module.

Let $\bm l =(l_1,\dots,l_{m+n-1})$ be a collection of non-negative integers. The \emph{weight function} is a vector $w^{\bs s}(\bs t,\bs z)$ in $L(\bm\lambda,\bm z)$ depending on variables $\bm{t}=(t_{1}^{(1)},\dots,t_{l_1}^{(1)};\dots;t_{1}^{(m+n-1)},\dots,t_{l_{m+n-1}}^{(m+n-1)})$ and parameters $\bs z=(z_1,\dots,z_p)$. The weight function $w^{\bs s}(\bs t,\bs z)$ is constructed as follows, see \cite[Section 5.2]{BR}.

Set $l^{<a}=l_1+\cdots+l_{a-1}$, $a=1,\dots,m+n$. Note that $l=l^{<m+n}$. Consider a series in $l$ variables $\bm t$ with coefficients in $\YglMN$:
\begin{align*}
\mathbb B^{\bm s}_{l}(\bm t)= (\mathrm{str}^{\vphantom1}_{12\cdots l}\otimes \id)\Big(&\mc L^{(1,l+1)}(t_1^{(1)})\cdots\mc L^{(l,l+1)}(t_{l_{m+n-1}}^{(m+n-1)})\\
&\times 
\mathfrak R^{(1\dots l)}(\bm t){E^{\bm s}_{m+n,m+n-1}}^{\otimes l_{m+n-1}}\otimes\cdots\otimes {E^{\bm s}_{21}}^{\otimes l_1}
\otimes 1\Big),
\end{align*}
where
\beq\label{norm R}
\mathfrak R^{(1\dots l)}(\bm t)=\prod_{a<b}\mathop{\overrightarrow\prod}\limits_{1\lle j\lle l_b}\mathop{\overleftarrow\prod}\limits_{1\lle i\lle l_a}\frac{t_j^{(b)}-t_i^{(a)}}{t_j^{(b)}-t_i^{(a)}+s_bh}\,  R^{(l^{<b}+j,l^{<a}+i)}(t_j^{(b)}-t_i^{(a)})
\eeq
and the first product in \eqref{norm R} runs over $1\lle a<b\lle m+n-1$.

The weight function $w^{\bs s}(\bs t,\bs z)\in L(\bla,\bm z)$ is given by
\[
w^{\bs s}(\bs t,\bs z)=\mathbb B^{\bm s}_{l}(\bm t)\big(v_1^{\bm s}\otimes\cdots\otimes v_p^{\bm s}\big)\, .
\]

\begin{eg}
Let $m+n=2$ and $\bm t=(t_1,\dots,t_l)$, then
\beq\label{BV 2}
w^{\bm s}(\bm t,\bm z)=(-1)^{l|2|}\mc L_{12}^{\bm s}(t_1)\cdots \mc L_{12}^{\bm s}(t_l)\big(v_1^{\bm s}\otimes\cdots\otimes v_p^{\bm s}\big)
\eeq
is an example of the weight function.
\qed
\end{eg}

The following theorem is known.
\begin{thm}[\cite{BR}]\label{thm BR 08}
Suppose that $\bm{\lambda}$ is a sequence of polynomial $\glMN$ weights and $\bs t$ a solution of the BAE associated to $\bs s$, $\bs z$, $\bm\lambda$, and $\bm{l}$. If the vector $w^{\bs s}(\bs t,\bs z)\in L(\bla,\bm z)$ is well-defined and non-zero, then $w^{\bs s}(\bs t,\bs z)\in L(\bla,\bm z)$ is an eigenvector of the transfer matrix $\cT(x)$, $\mc T(x)w^{\bs s}(\bs t,\bs z)=\mc E(x)w^{\bs s}(\bs t,\bs z)$, where the eigenvalue $\mc E(x)$ is given by
\beq\label{hamiltonian}
\mc E(x)=\sum_{a=1}^{m+n}s_a\prod_{k=1}^p\frac{x-z_k+s_a\la_a^{(k,\bm s)}h}{x-z_k}\prod_{j=1}^{l_{a-1}}\frac{x-t_j^{(a-1)}+s_ah}{x-t_{j}^{(a-1)}}\prod_{j=1}^{l_a}\frac{x-t_j^{(a)}-s_ah}{x-t_{j}^{(a)}}\ .
\eeq
\end{thm}

Note that the eigenvalue $\mc E(x)$ depends on the parameters $\bm t$, $\bs s$, $\bs z$, and $\bm\lambda$. We drop this dependence for our notation.

If $\bs t$ is a solution of the BAE associated to $\bs s$, $\bs z$, $\bm\lambda$, and $\bm{l}$, then the value of the weight function $w^{\bs s}(\bs t,\bs z)$ is called the \emph{Bethe vector}.

We have the following standard statement regarding to Bethe vectors, c.f. \cite[Proposition 6.2]{MTV} and \cite[Theorem 4.3]{MVY}.

\begin{prop}\label{prop BV singular}
The Bethe vector $w^{\bs s}(\bs t,\bs z)$ is a $\gl_{m|n}$ $\bm s$-singular vector of weight $\la^{(\bm s,\infty)}$.
\end{prop}
\begin{proof}
Clearly, the Bethe vector $w^{\bs s}(\bs t,\bs z)$ is a vector of weight $\la^{(\bm s,\infty)}$. We then show that $w^{\bs s}(\bs t,\bs z)$ is $\gl_{m|n}$ $\bm s$-singular.

We show it for the case of $m=n=1$ with the standard parity $\bm s_0$ in Section \ref{sec 7.2}. The general case follows from a similar computation using a combination of nested Bethe ansatz, as in \cite[Section 4]{BR}, and induction on $m+n$, see e.g. \cite[Proposition 6.2]{MTV}.
\end{proof}

\subsection{Sequences of polynomials}
We use the following convenient notation. We say that a sequence $\bm z=(z_1,\dots,z_p)$ of complex numbers is \emph{$h$-generic} if $z_i-z_j\notin h\Z$ for all $1\lle i<j\lle p$.

Let $\bm\lambda=(\lambda^{(1)},\dots,\lambda^{(p)})$ be a sequence of polynomial $\gl_{m|n}$ weights. Let $\bs z=(z_1,\dots,z_p)$ be an $h$-generic sequence of complex numbers. Fix a parity sequence $\bs s\in S_{m|n}$.

Define a sequence of polynomials $\bm{T}^{\bm{s}}=(T_1^{\bm{s}},\dots,T_{m+n}^{\bm{s}})$ associated to $\bm{s}$, $\bm{\lambda}$ and $\bm{z}$, 
\begin{equation}\label{eq functions T}
T_i^{\bm{s}}(x)=\prod_{k=1}^{p}\prod_{j=1}^{\la_i^{(k,\bm s)}}(x-z_k+s_ijh),\qquad i=1,\dots,m+n.
\end{equation}
Note that  $T^{\bs s}_i(T^{\bm{s}}_{i+1})^{-s_is_{i+1}}$ is a polynomial for all $i=1,\dots,m+n-1$.

Let $\bs l=(l_1,\dots,l_{m+n-1})$ be a sequence of non-negative integers. 

Let $\bm{t}=(t_{1}^{(1)},\dots,t_{l_1}^{(1)};\dots;t_{1}^{(m+n-1)},\dots,t_{l_{m+n-1}}^{(m+n-1)})$ be a sequence of complex numbers. Define a sequence of polynomials $\bm{y}=(y_1,\dots,y_{m+n-1})$ by 
\beq \label{y fun}
y_i(x)=\prod_{j=1}^{l_i}(x-t_j^{(i)}),\qquad i=1,\dots,m+n-1.
\eeq 
We say the \emph{sequence of polynomials $\bm{y}$ represents $\bm{t}$}. We have $\deg y_i=l_i$. 

We also set $y_0(x)=y_{m+n}(x)=1$.

If $\bm{t}$ is a solution of the BAE associated to $\bm{s}$, $\bm{z}$, $\bm{\lambda}$, and $\bm{l}$, then the eigenvalue  $\mc E(x)$ of the transfer matrix $\mc T(x)$ acting on the Bethe vector $w^{\bs s}(\bs t,\bs z)$, see \eqref{hamiltonian}, can be written in terms of $\bm y$ and $\bm T^{\bm s}$. Namely, we have
\beq\label{glmn eigenvalue}
\mc E(x)=\mc E_{\bm y}(x)=\sum_{a=1}^{m+n}s_a\,\frac{T_a^{\bm s}}{T_a^{\bm s}[s_a]}\frac{y_{a-1}[-s_a]}{y_{a-1}}\frac{y_a[s_a]}{y_a}\,.
\eeq

We do not consider zero polynomials $y_i(x)$ and
do not distinguish between polynomials $y_i(x)$ and $cy_i(x)$, $c\in \C^{\times}$. 
Hence, a sequence $\bs y$ defines a point in $\big(\bP(\C[x])\big)^{m+n-1}$, the direct product  of $m+n-1$ copies of the projective 
space associated to the vector space of polynomials. 
\medskip

We say that a sequence of polynomials $\bs y$ is \emph{generic with respect to $\bs s$, $\bm\lambda$, and $\bs z$} if it satisfies the following conditions:
\begin{enumerate}
\item if $s_is_{i+1}=1$, then $y_i$ has only simple roots and $y_i$ has no common roots with the polynomial $y_i[1]$;
\item the polynomial $y_i$ has no common roots with polynomials $y_{i-1}$, $y_{i-1}[-s_i]$, and $y_{i+1}[s_{i+1}]$;
\item all roots of $y_i$ are different from the roots of polynomial $T^{\bs s}_i(T^{\bs s}_{i+1})^{-s_is_{i+1}}$,
\end{enumerate}
for $i=1,\dots,m+n-1$.

Not all solutions of the BAE correspond to generic sequences of polynomials. For instance, if $m=2$, $n=p = 0$,
and $l$ is even, then $t_1 =\dots= t_l = 0$ is a solution of the BAE.

\section{Reproduction procedures for $\gl_2$ and $\gl_{1|1}$}
\label{sec rep pro gl2 gl11}
In this section, we recall the reproduction procedure for the XXX model associated to $\gl_2$ from \cite[Section 2]{MV03} and define its analogue for $\gl_{1|1}$. We define a rational difference operator associated to a solution of BAE. We also show that the reproduction procedure does not alter the rational difference operator and the corresponding eigenvalues obtained from Theorem \ref{thm BR 08}.

\subsection{Reproduction procedure for $\gl_2$.}
Set $m=2$ and $n=0$. We have the following identifications $\mathrm Y(\gl_{2|0})\cong\mathrm Y(\gl_{0|2})\cong\mathrm Y(\gl_{2})$. Let $\bm\lambda=(\lambda^{(1)},\dots,\lambda^{(p)})=\left((a_1,b_1),\dots,(a_p,b_p)\right)$ be a sequence of polynomial $\gl_2$ weights. We have $a_k, b_k \in\Z$,  $a_k\gge b_k\gge 0$, $k=1,\dots,p$.  Let $\bm z=(z_1,\dots,z_p)$ be an $h$-generic sequence of complex numbers. We have \[
T_1(x)=\prod_{k=1}^p\prod_{j=1}^{a_k}(x-z_k+jh),\qquad T_2(x)=\prod_{k=1}^p\prod_{j=1}^{b_k}(x-z_k+jh).
\]
Let $a=\deg T_1$ and $b=\deg T_2$.

Give a non-negative integer $l$ and variables $\bm t=(t_1,\dots,t_l)$. The BAE associated to $\bm\lambda$, $\bm z$, and $l$ is simplified to
\begin{equation}\label{eq gl2 XXX BAE}
\prod_{k=1}^p \frac{t_j-z_k+a_kh}{t_j-z_k+b_kh}\prod_{i=1,i\ne j}^l \frac{t_j-t_i-h}{t_j-t_i+h}=1,\qquad j=1,\dots,l.
\end{equation}

It is known that the BAE \eqref{eq gl2 XXX BAE} can be reformulated in terms of discrete Wronskian. Moreover, starting from a generic solution of BAE, one can construct a family of new solutions of the BAE in the following way. 
\begin{lem}[\cite{MV03}]\label{lem gl2 populations} 
Let $y$ be a  polynomial of degree $l$ which is generic with respect to $\bm{\lambda}$ and $\bm z$.
\begin{enumerate}
\item The polynomial $y\in\C[x]$ represents a solution of the BAE \eqref{eq gl2 XXX BAE} associated to $\bm\lambda$, $\bm z$ and $l$, if and only if there exists a polynomial $\tl{y}\in\C[x]$, such that
\begin{equation}\label{eq gl2 wronski}
\Wr^+(y,\tl{y})=T_1T_2^{-1}.
\end{equation}
\item  If $\tl{y}$ is generic, then $\tl{y}$ represents a solution of the BAE associated to $\bm\lambda$, $\bm z$ and $\tl{l}$, where $\tl{l}=\deg \tl{y}$.\qed
\end{enumerate}		
\end{lem}

Almost all $\tl{y}$ are generic with respect to $\bm{\lambda}$ and $\bm z$, and therefore by Lemma \ref{lem gl2 populations}  represent solutions of the BAE \eqref{eq gl2 XXX BAE}. Thus, from one solution of the BAE, we obtain a family of new solutions. Following the terminology of \cite{MV03}, we call this construction the \emph{$\gl_2$ reproduction procedure}.

Let $P_y$ be the closure of the set containing $y$ and all $\tl{y}$ as in Lemma \ref{lem gl2 populations}  in $\mathbb P(\C[x])$. We call $P_y$ the \emph{$\gl_2$ population originated at $y$}. The population $P_y$ can be identified with the projective line $\C\mathbb P^1$ through the correspondence $c_1y+c_2\tl y\mapsto (c_1:c_2)$.

The weight at infinity associated to the data $\bla$ and $l$ is given by $\lambda^{(\infty)}=(a-l,b+l)$. 
Suppose that the weight $\lambda^{(\infty)}$ is dominant, namely $2l\lle a-b$.
If $\tl l\ne l$, then the weight at infinity associated to $\bm \lambda$ and $\tl{l}$ is
\[\tl{\lambda}^{(\infty)}=(a-\tl{l},b+\tl{l})=(b+l-1,a-l+1)=s\cdot \lambda^{(\infty)},\]
where $s\in\mathfrak S_2$ is the non-trivial element in the Weyl group of $\gl_2$, and the dot denotes the shifted action.

Let $\tl y=\prod_{r=1}^{\tl{l}}(x- \tl{t}_r)$ and $\tl {\bm{t}}=(\tl{t}_1,\dots,\tl{t}_{\tl{l}})$. If $y$ is generic, then by Lemma \ref{lem gl2 populations}, $\tl {\bm t}$ is a solution of the BAE \eqref{eq gl2 XXX BAE} with $l$ replaced by $\tl l$. By Proposition \ref{prop BV singular}, the value of the weight function $w(\tl{\bm t},\bm z)$ is a singular vector. At the same time, $\tl{\lambda}^{(\infty)}$ is not dominant and therefore  $w(\tl{\bm t},\bm z)=0$ in $L(\bm \lambda)$. So, in a $\gl_2$ population only the unique polynomial (the one of the smallest degree) corresponds to an actual eigenvector in $L(\bm\lambda)$.

The eigenvalues corresponding to the solutions $y$ and $\tl y$, see \eqref{glmn eigenvalue}, are given by
\[
\mc E(x)=\frac{T_1y[1]}{T_1[1]y}+\frac{T_2y[-1]}{T_2[1]y},\quad \tl{\mc E}(x)=\frac{T_1\tl y[1]}{T_1[1]\tl y}+\frac{T_2\tl y[-1]}{T_2[1]\tl y}.
\]

\begin{lem}\label{lem eigen gl2}
The eigenvalues $\mc E(x)$ and $\tl{\mc E}(x)$ are the same.
\end{lem}
\begin{proof}
Note that
\begin{align*}
\tl{\mc E}(x)-\mc E(x)=\ &\frac{\Wr^+(y,\tl y)[1]}{y\tl y}\frac{T_1}{T_1[1]}-\frac{\Wr^+(y,\tl y)}{y\tl y}\frac{T_2}{T_2[1]}.
\end{align*}
By \eqref{eq gl2 wronski}, we have
\[
\frac{\Wr^+(y,\tl y)}{\Wr^+(y,\tl y)[1]}=\frac{T_1T_2[1]}{T_2T_1[1]}.
\]
Therefore the lemma follows.
\end{proof}
This fact can be reformulated in the following form.

Define a difference operator
$$
\cD(y)=\Big(1-\frac{T_1y[1]}{T_1[1]y} \,\tau\Big)\Big(1-\frac{T_2y[-1]}{T_2[1]y}\,\tau\Big).
$$
The operator $\cD(y)$ does not depend on a choice of polynomial $y$ in a population, $\cD(y)=\cD(\tl y)$.

\subsection{Reproduction procedure for $\gl_{1|1}$.}
\label{sec repro pro gl11} 
Set $m=n=1$. We have $ S_{1|1}=\{(1,-1),(-1,1)\}$. Let $\bm s$ and $\tl{\bm s}=\bm s^{[1]}$ be two different parity sequences in $ S_{1|1}$. Let $\bm\lambda=(\lambda^{(1)},\dots,\lambda^{(p)})$ be a sequence of polynomial $\gl_{1|1}$ weights. For each $k=1,\dots,p$, let us write $(\lambda^{(k)})^{\bm s}_{[\bm s]}=(a_k,b_k)$, where $a_k,b_k\in\Z_{\gge 0}$ and if $a_k=0$ then $b_k=0$. Note that $\lambda^{(k)}$ is atypical if and only if $a_k+b_k = 0$. 
Let $\bm z=(z_1,\dots,z_p)$ be an $h$-generic sequence of complex numbers.

Let $$\tl{a}_k=\begin{cases} b_k+1 \ \ &{\rm if}\ \ a_k+b_k\neq 0,
\\ 0\ &{\rm if}\ \ a_k+b_k=0,\end{cases} \qquad  \tl{b}_k=\begin{cases} a_k-1 \ \ &{\rm if}\ \ a_k+b_k\neq 0,
\\ 0\ &{\rm if}\ \ a_k+b_k=0.\end{cases}  $$

Equation \eqref{eq functions T} becomes
\[
T_1^{\bm{s}}=\prod_{k=1}^p\prod_{j=1}^{a_k}(x-z_k+s_1jh),\quad T_2^{\bm{s}}=\prod_{k=1}^p\prod_{j=1}^{b_k}(x-z_k+s_2jh),
\]
\[
T_1^{\tl{\bm{s}}}=\prod_{k=1}^p\prod_{j=1}^{\tl a_k}(x-z_k+\tl s_1jh),
\quad
T_2^{\tl{\bm{s}}}=\prod_{k=1}^p\prod_{j=1}^{\tl b_k}(x-z_k+\tl s_2jh).
\]
Let $a=\deg T_1^{\bm s}$, $b=\deg T_2^{\bm s}$. Similarly, let $\tl{a}=\deg T_1^{\tl{\bm s}}$, $\tl{b}=\deg T_2^{\tl{\bm s}}$. Suppose the number of typical weights in $\bla$ is $r$, then $\tl{a}=b+r$ and $\tl{b}=a-r$. 

Let $l$ be a non-negative integer. Let $\bm t=(t_1,\dots,t_l)$ be a collection of variables. The Bethe ansatz equation associated to $\bm s$, $\bm\lambda$, $\bm z$, and $l$, is given as follows,
\beq\label{eq gl11 XXX BAE}
\prod_{\substack{k=1 \\ a_k+b_k\neq 0}}^p\frac{t_j-z_k+s_1a_kh}{t_j-z_k+s_2b_kh}=1, \qquad j=1,\dots,l.
\eeq

The Bethe ansatz equation \eqref{eq gl11 XXX BAE} can be rewritten in the form		
\[
\varphi^{\bm s}(t_j)-\psi^{\bm s}(t_j)=0,
\]
where
\[
\varphi^{\bm s}=\prod_{\substack{k=1 \\ a_k+b_k\neq 0}}^p(x-z_k+ s_1a_kh),\qquad \psi^{\bm s}=\prod_{\substack{k=1 \\ a_k+b_k\neq 0}}^p(x-z_k+ s_2b_kh).
\] 
Note that $\varphi^{\bm s}=\psi^{\tl{\bm s}}[-s_1]$ and $\psi^{\bm s}=\varphi^{\tl{\bm s}}[-s_1]$. Thus, in the case of $\gl_{1|1}$, the BAEs \eqref{eq gl11 XXX BAE} associated to $\bm s$ and $\tl{\bm s}$ coincide up to a shift. 

\medskip

We call a sequence of polynomial $\gl_{1|1}$ weights $\bla$ {\it typical} if at least one of the weights $\lambda^{(k)}$ is typical. 
Note that $\bm \lambda$ is typical if and only if $a+b\neq 0$. In other words, $\bm \lambda$ is typical if and only if $T_1^{\bm{s}}T_2^{\bm{s}}\ne 1$.

The BAE \eqref{eq gl11 XXX BAE} is reformulated as follows, c.f. \cite[equation (A.12)]{GLM}.
\begin{lem}\label{lemma gl1|1 repro pro} 
Let $y$ be a polynomial of degree $l$. Let $\bm \lambda$ be typical.
\begin{enumerate}
\item The polynomial $y$ represents a solution of the BAE \eqref{eq gl11 XXX BAE} associated to $\bm s$, $\bm z$, $\bm\lambda$, and $l$, if and only if there exists a polynomial $\tl{y}$, such that 
\begin{equation}\label{eq gl11 wronski}
y\cdot\tl{y}[-s_1]=\varphi^{\bm s}-\psi^{\bm s}.
\end{equation}
\item  The polynomial $\tl{y}$ represents a solution of the BAE \eqref{eq gl11 XXX BAE} associated to $\tl{\bm{s}}$, $\bm z$, $\bm{\lambda}$, and $\tl{l}$, where $\tl{l}=\deg \tl{y}=r-1-l$.
\qed
\end{enumerate}
\end{lem}
For each solution $y$, we can construct exactly one solution $\tl{y}$. We call this construction the \emph{$\gl_{1|1}$ reproduction procedure}.

The set $P_y$ consisting of $y$ and $\tl{y}$ is called the \emph{$\gl_{1|1}$ population originated at $y$}.

The weight at infinity associated to $\bm s, \bm \lambda$, and $l$ is $\lambda^{(\bs s,\infty)}_{[\bs s]}=(a-l,b+l)$. The weight at infinity associated to $\tl{\bs s}, \bs \lambda$ and $\tl{l}$ is $\tl{\lambda}^{(\tl{\bs s},\infty)}_{[\tl{\bm{s}}]}=(\tl{a}-\tl{l},\tl{b}+\tl{l})= (b+l+1,a-l-1)$. Thus we have $\lambda^{(\bs s,\infty)}=\tl{\lambda}^{(\tl{\bs s},\infty)}+\alpha^{\bs s}$.
In particular, in contrast to the case of $\gl_2$, both $y$ and $\tl{y}$ correspond to actual eigenvectors of the transfer matrix.

If $\bs \la$ is not typical, then all participating representations are one-dimensional, where the situation is trivial. In particular, we have $y(x)=1$. We do not discuss this case.

\subsection{Motivation for $\gl_{1|1}$ reproduction procedure}
Suppose $y$ and $\tl y$ are in the same $\gl_{1|1}$ population as in Section \ref{sec repro pro gl11}. Parallel to the $\gl_2$ reproduction procedure, we show that the eigenvalues of transfer matrix corresponding to the Bethe vectors obtained from polynomials $y$ and $\tl y$ coincide.

Let $y=\prod_{r=1}^l(x-t_r)$,  $\tl{y}=\prod_{r=1}^{\tl{l}}(x-\tl{t}_r)$. Let 
$\bm{t}=(t_1,\dots,t_l)$, $\tl{\bm{t}}=(\tl{t}_1,\dots,\tl{t}_{\tl{l}})$. By Theorem \ref{thm BR 08} and \eqref{glmn eigenvalue}, we have $\cT(x) w^{\bs s}(\bs t,\bs z)=\mc E(x)w^{\bs s}(\bs t,\bs z)$ and 
$\cT(x)w^{\tl{\bs s}}(\tl{\bs t},\bs z)=\tl{\mc E}(x)w^{\tl{\bs s}}(\tl{\bs t},\bs z)$, where
\be
\mc E(x)=s_1\frac{T_1^{\bm s}y[s_1]}{T_1^{\bm s}[s_1]y}+s_2\frac{T_2^{\bm s}y[-s_2]}{T_2^{\bm s}[s_2]y},\qquad 
\tl{\mc E}(x)=\tl s_1\frac{T_1^{\tl{\bm s}}y[\tl s_1]}{T_1^{\tl{\bm s}}[\tl s_1]y}+\tl s_2\frac{T_2^{\tl{\bm s}}y[-\tl s_2]}{T_2^{\tl{\bm s}}[\tl s_2]y}.
\ee

\begin{lem}\label{lem gl11 eigen}
The eigenvalues $\mc E(x)$ and $\tl{\mc E}(x)$ of transfer matrix are the same.
\end{lem}
\begin{proof}
By \eqref{eq gl11 wronski}, we have
\begin{align*}
\mc E(x)=s_1\frac{y[s_1]}{y}(\varphi^{\bm s}-\psi^{\bm s})\prod_{\substack{k=1 \\ a_k+b_k\ne 0}}^p(x-z_k)^{-1}=s_1y[s_1]\tl y[-s_1]\prod_{\substack{k=1 \\ a_k+b_k\ne 0}}^p(x-z_k)^{-1},
\end{align*}
and
\begin{align*}
\tl{\mc E}(x)= s_1\frac{\tl y[-s_1]}{\tl y}(\varphi^{\bm s}[s_1]-\psi^{\bm s}[s_1])\prod_{\substack{k=1 \\ a_k+b_k\ne 0}}^p(x-z_k)^{-1}
= s_1y[s_1]\tl y[-s_1]\prod_{\substack{k=1 \\ a_k+b_k\ne 0}}^p(x-z_k)^{-1}.
\end{align*}
Therefore the lemma follows.
\end{proof}

Define a rational difference operator:
$$
\cR^{\bs s}(y)=\Big(1-\frac{T_1^{\bm s}y[s_1]}{T_1^{\bm s}[s_1]y} \,\tau\Big)^{s_1}\Big(1-\frac{T_2^{\bm s}y[-s_2]}{T_2^{\bm s}[s_2]y}\,\tau\Big)^{s_2}.
$$	
It is clear that $\cR^{\bs s}(y)=1$ if $\bs\la$ is not typical.

We have the following lemma.

\begin{lem}\label{lem 4.5}
If $\bs\la$ is typical, then $\cR^{\bs s}(y)$ is a $(1|1)$-rational difference operator. Moreover, this $(1|1)$-rational difference operator is independent of a choice of a polynomial in a population, $\cR^{\bs s}(y)=\cR^{\tl{\bs s}}(\tl{y})$. 
\end{lem}
\begin{proof}
The lemma is proved by a direct computation using Lemma \ref{eq relations diff} and \eqref{eq gl11 wronski}.
\end{proof}

\section{Reproduction procedure for $\gl_{m|n}$}\label{sec glmn rep pro}
We define the reproduction procedure and the populations in the general case.

\subsection{Reproduction procedure}\label{subsec glmn rep pro}
Let $\bm{s}\in S_{m|n}$ be a parity sequence. Let $\bm{\lambda}=(\lambda^{(1)},\dots,\lambda^{(p)})$ be a sequence of polynomial $\gl_{m|n}$ weights. Let $\bm{z}=(z_1,\dots,z_p)$ be an $h$-generic sequence of complex numbers. Let $\bm{T}^{\bm{s}}$ be a sequence of polynomials associated to $\bm{s}$, $\bm{\lambda}$, and $\bm{z}$, see \eqref{eq functions T}. 
If $s_i\ne s_{i+1}$, we also set
\[
\varphi_i^{\bm s}=\prod_{\substack{k=1 \\ \la_{i}^{(k,\bm s)}+\la_{i+1}^{(k,\bm s)}\ne 0}}^p(x-z_k+s_i\la_{i}^{(k,\bm s)}h),\quad\psi_i^{\bm s}=\prod_{\substack{k=1 \\ \la_{i}^{(k,\bm s)}+\la_{i+1}^{(k,\bm s)}\ne 0}}^p(x-z_k+s_{i+1}\la_{i+1}^{(k,\bm s)}h).
\]
Let $\bm{l}=(l_1,\dots,l_{m+n-1})$ be a sequence of non-negative integers.

For $i\in \{1,\dots, m+n-1\}$, set $\bm{s}^{[i]}=(s_1,\dots,s_{i+1},s_{i},\dots,s_{m+n})$.
Set $y_0=y_{m+n}=1$.

For $g_1,g_2\in \bK$, we also use the notation
\[
\Wr^{s_i}(g_1,g_2)=g_1g_2[-s_i]-g_2g_1[-s_i].
\]

We now reformulate the BAE \eqref{eq BAE XXX} which allows us to construct a family of new solutions. 

\begin{thm}\label{thm repro pro glmn}  
Let $\bm{y}=(y_1,\dots,y_{m+n-1})$ be a sequence of polynomials generic with respect to $\bm{s}$, $\bm{\lambda}$, and $\bm{z}$, such that $\deg y_k=l_k$, $k=1,\dots,m+n-1$.
\begin{enumerate}
\item The sequence $\bm{y}$ represents a solution of the BAE \eqref{eq BAE XXX} associated to $\bm{s}$, $\bm{z}$, $\bm{\lambda}$, and $\bm{l}$, if and only if for each $i=1,\dots,m+n-1$, there exists a polynomial $\tl{y}_i$, such that
\begin{align}
&\Wr^{s_i}\left(y_i,\tl{y}_i\right)=T^{\bm{s}}_i\left(T^{\bm{s}}_{i+1}\right)^{-1}y_{i-1}[-s_i]y_{i+1}, &\mbox{if}\;s_i=s_{i+1},\label{bosonic rp}\\
&y_i\,\tl{y}_i[-s_i]=\varphi_i^{\bm s}y_{i-1}[-s_i]y_{i+1}-\psi_{i}^{\bm s}y_{i-1}y_{i+1}[-s_i], &\mbox{if}\;s_i\neq s_{i+1}.\label{fermionic rp}
\end{align}
\item Let $i\in \{1,\dots, m+n-1\}$ be such that $\tl y_i\neq 0$. If $\bm{y}^{[i]}=(y_1,\dots,\tl{y}_i,\dots,y_{m+n-1})$ 
is generic with respect to $\bm{s}^{[i]}$, $\bm{\lambda}$, and $\bm{z}$, then $\bm{y}^{[i]}$ represents a solution of the BAE associated to $\bm{s}^{[i]}$, $\bm{\lambda}$, $\bm{z}$, and $\bm{l}^{[i]}$, where
$\bm{l}^{[i]}=(l_1,\dots,\tl{l}_i,\dots,l_{m+n-1})$, $\tl{l}_i=\deg \tl{y}_i$.
\end{enumerate}
\end{thm}

\begin{proof} 
Part (i) follows from Lemmas \ref{lem gl2 populations} and \ref{lemma gl1|1 repro pro}.

Now we consider Part (ii). Let $y_r=\prod_{j=1}^{l_r}(x-t^{(r)}_j)$ and $\tl{y}_r= \prod_{j=1}^{\tl{l}_r}(x-\tl{t}^{\;(r)}_j)$, $r=1,\dots,m+n-1$. Let $\bs t=(t_j^{(r)})_{r=1,\dots, m+n-1}^{j=1,\dots,l_r}$ and $\tl{\bs t}=(\tl {t}_j^{\;(r)})_{r=1,\dots, m+n-1}^{j=1,\dots,\tl{l}_r}$, where we set $l_r=\tl{l}_r, t_j^{(r)}=\tl {t}_j^{\;(r)}$ if $r\neq i$. 

The sequence ${\bs t}$ satisfies the BAE associated to $\bm{s}$, $\bm{\lambda}$, $\bm{z}$, and $\bm{l}$. We prove that $\tl{\bs t}$ satisfies the BAE associated to $\bm{s}^{[i]}$, $\bm{\lambda}$, $\bm{z}$, and $\bm{l}^{[i]}$. Clearly, the BAEs for $\tl {\bs t}$ and ${\bs t}$ related to $t^{(r)}_j$ with $|r-i|>1$ are the same. On the other hand, the BAE for $\tl{\bs t}$ related to $\tl{t}^{\;(i)}_j$ holds by Lemmas \ref{lem gl2 populations}  and \ref{lemma gl1|1 repro pro}.  We only need to establish the BAE for $\tl{\bs t}$ related to $t^{(i-1)}_j$ and $t^{(i+1)}_j$. We have two main cases depending on the sign of $s_is_{i+1}$.

Suppose $s_{i}=s_{i+1}$. Dividing \eqref{bosonic rp} by $y_i[-s_i]\tl{y}_i[-s_i]$ and evaluating at $x=t^{(i-1)}_j-s_ih$ and $x=t^{(i+1)}_j$, we obtain 
\[
\prod_{a=1}^{l_i}\frac{t^{(i-1)}_j-t^{(i)}_a}{t^{(i-1)}_j-t^{(i)}_a-s_ih}=\prod_{a=1}^{\tl l_i}\frac{t^{(i-1)}_j-\tl t^{(i)}_a}{t^{(i-1)}_j-\tl t^{(i)}_a-s_ih}\, ,
\]
\[
\prod_{a=1}^{l_i}\frac{t^{(i+1)}_j-t^{(i)}_a+s_ih}{t^{(i+1)}_j-t^{(i)}_a}=\prod_{a=1}^{\tl l_i}\frac{t^{(i+1)}_j-\tl t^{(i)}_a+s_ih}{t^{(i+1)}_j-\tl t^{(i)}_a}\, .
\]
Thus, the BAE for $\tl{\bs t}$ related to $t^{(i\pm1)}_j$ follows from the BAE for $\bs t$ related to $t^{(i\pm1)}_j$.

If $s_{i}=-s_{i+1}$, then the argument depends on $s_{i-1}$, $s_{i+2}$. Here we only treat the case of $s_{i-1}=-s_{i}$. All other cases are similar, we omit further details. 

We prove the BAE for $\tl{\bs t}$ related to $t^{(i-1)}_j$, which has the form
\beq\label{i-1 bae}
\frac{\varphi_{i-1}^{\bm{s}^{[i]}}(t_j^{(i-1)})}{\psi_{i-1}^{\bm{s}^{[i]}}(t_j^{(i-1)})}\cdot\frac{y_{i-2}(t_j^{(i-1)}+s_{i-1}h)}{y_{i-2}(t_j^{(i-1)})}\cdot\frac{y_{i-1}(t_j^{(i-1)}-s_{i-1}h)}{y_{i-1}(t_j^{(i-1)}+s_{i+1}h)}\cdot\frac{\tl y_i(t_j^{(i-1)})}{\tl y_i(t_j^{(i-1)}+s_ih)}=-1\, .
\eeq
Substituting $x=t_j^{(i-1)}-s_ih$ and $x=t_j^{(i-1)}$ to \eqref{fermionic rp} and dividing, we get
\beq\label{exp 1}
\frac{\tl y_i(t_j^{(i-1)})}{\tl y_i(t_j^{(i-1)}+s_ih)}=-\frac{\psi_i^{\bm s}(t_j^{(i-1)}-s_ih)y_{i-1}(t_j^{(i-1)}+s_{i+1}h)y_i(t_j^{(i-1)})}{\varphi_i^{\bm s}(t_j^{(i-1)})y_{i-1}(t_j^{(i-1)}-s_{i-1}h)y_i(t_j^{(i-1)}-s_ih)}\, .
\eeq
Changing $i$ in \eqref{fermionic rp} to $i-1$ (recall $s_{i-1}=-s_{i}$) and substituting $x=t_j^{(i-1)}$, we have
\beq\label{exp 2}
\frac{\varphi_{i-1}^{\bm s}(t_j^{(i-1)})y_{i-2}(t_j^{(i-1)}+s_{i-1}h)y_i(t_j^{(i-1)})}{\psi_{i-1}^{\bm s}(t_j^{(i-1)})y_{i-2}(t_j^{(i-1)})y_i(t_j^{(i-1)}-s_ih)}=1\, .
\eeq
Equation \eqref{i-1 bae} follows from \eqref{exp 1}, \eqref{exp 2}, and the equality
\[
\frac{\varphi_{i-1}^{\bm{s}^{[i]}}(t_j^{(i-1)})}{\psi_{i-1}^{\bm{s}^{[i]}}(t_j^{(i-1)})}=\frac{\varphi_{i-1}^{\bm s}(t_j^{(i-1)})\varphi_i^{\bm s}(t_j^{(i-1)})}{\psi_{i-1}^{\bm s}(t_j^{(i-1)})\psi_i^{\bm s}(t_j^{(i-1)}-s_ih)}\, .\qedhere
\]
\end{proof}
\begin{rem}
Suppose $s_i\ne s_{i+1}$. It is not hard to see that if $\varphi_i^{\bm s}y_{i-1}[-s_i]y_{i+1}$ and $\psi_{i}^{\bm s}y_{i-1}y_{i+1}[-s_i]$ in \eqref{fermionic rp} have common roots, then $\bm{y}^{[i]}$ 
is not generic with respect to $\bm{s}^{[i]}$, $\bm{\lambda}$, and $\bm{z}$.
\end{rem}

If $s_i=s_{i+1}$, then starting from a solution of the BAE
we construct a family of new solutions represented by sequences $\bm{y}^{[i]}$. Here we use \eqref{bosonic rp} and the parity sequence remains unchanged. We call this construction the \emph{bosonic reproduction procedure in $i$-th direction}. 

If $\varphi_i^{\bm s}y_{i-1}[-s_i]y_{i+1}\ne\psi_{i}^{\bm s}y_{i-1}y_{i+1}[-s_i]$, 
then starting from a solution of the BAE we construct a single new solution represented by $\bm{y}^{[i]}$. We use \eqref{fermionic rp} and the parity sequence changes from $\bs s$ to $\bs s^{[i]}$. We call this construction the \emph{fermionic reproduction procedure in $i$-th direction}. 

From the very definition of the fermionic reproduction procedure, $(\bm{y}^{[i]})^{[i]}=\bm{y}$. 

If $\bm{y}^{[i]}$ is generic with respect to $\bm{s}^{[i]}$, $\bm{\lambda}$, and $\bm{z}$, then by Theorem \ref{thm repro pro glmn} we can apply the reproduction procedure again.

Let 
\be 
P_{(\bm{y},\bs{s} )}\subset \big(\mathbb{P}(\C[x])\big)^{m+n-1}\times S_{m|n}
\ee 
be the closure of the set of all pairs $(\tilde{\bs y}, \tilde{\bs s})$ obtained from the initial pair $(\bs y,\bs s)$ by repeatedly applying all possible reproductions. We call $P_{(\bs y,\bs s)}$ the \emph{$\gl_{m|n}$ population of solutions of the BAE associated to $\bm{s}$, $\bm{z}$, and $\bm{\lambda}$ , originated at $\bm{y}$}.
By definition, $P_{(\bs y,\bs s)}$ is a disjoint union over parity sequences, 
\[P_{(\bs y,\bs s)}=\bigsqcup_{\tl{\bm{s}}\in S_{m|n}}P^{\tilde{\bs s}}_{(\bs y,\bs s)},  \qquad
P^{\tl{\bs s}}_{(\bs y,\bs s)}  = P_{(\bm{y},\bs{s} )} \cap \left(\big(\mathbb{P}(\C[x])\big)^{m+n-1}\times \{\tl{\bs s}\}\right). \]  

\subsection{Rational difference operator associated to population}
We define a rational difference operator  which does not change under the reproduction procedure.

Let $\bm{s}\in S_{m|n}$ be a parity sequence. Let $\bm{z}=(z_1,\dots,z_p)$ be an $h$-generic sequence of complex numbers. Let $\bm{\lambda}=(\lambda^{(1)},\dots,\lambda^{(p)})$ be a sequence of polynomial $\gl_{m|n}$ weights. The sequence $\bm{T}^{\bm{s}}=(T^{\bm{s}}_1,\dots,T^{\bm{s}}_{m+n})$ is given by \eqref{eq functions T}.

Let $\bm{y}=(y_1,\dots,y_{m+n-1})$ be a sequence of polynomials. Recall our convention that $y_0=y_{m+n}=1$. Define a rational difference operator $\cR^{\bm{s}}(\bs y)$ over $\bK=\C(x)$, 
\begin{equation}\label{eq rational diff oper glmn}
\cR^{\bm{s}}(\bs y)=\mathop{\overrightarrow\prod}\limits_{1\lle i\lle m+n}\Big(1-\frac{T_i^{\bm s}y_{i-1}[-s_i]y_i[s_i]}{T_i^{\bm s}[s_i]y_{i-1}y_i}\, \tau\Big)^{s_i}.
\end{equation}

The following theorem is the main result of this section. 
\begin{thm}\label{thm diffoper inv}
Let $P$ be a $\gl_{m|n}$ population. Then the rational difference operator $\cR^{\bm{s}}(\bs y)$ does not depend on the choice of ${\bs y}$ in $P$.
\end{thm}
\begin{proof}
We want to show
\begin{align*}
&\Big(1-\frac{T_i^{\bm s}y_{i-1}[-s_i]y_i[s_i]}{T_i^{\bm s}[s_i]y_{i-1}y_i}\, \tau\Big)^{s_i}\Big(1-\frac{T_{i+1}^{\bm s}y_{i}[-s_{i+1}]y_{i+1}[s_{i+1}]}{T_{i+1}^{\bm s}[s_{i+1}]y_{i}y_{i+1}}\, \tau\Big)^{s_{i+1}}\\
=&\ \Big(1-\frac{T_i^{\bm{s}^{[i]}}y_{i-1}[-s_{i+1}]\tl y_i[s_{i+1}]}{T_i^{\bm{s}^{[i]}}[s_{i+1}]y_{i-1}\tl y_i}\, \tau\Big)^{s_{i+1}}\Big(1-\frac{T_{i+1}^{\bm{s}^{[i]}}\tl y_{i}[-s_i]y_{i+1}[s_i]}{T_{i+1}^{\bm{s}^{[i]}}[s_i]\tl y_{i}y_{i+1}}\, \tau\Big)^{s_i}.
\end{align*} We have four cases, $(s_i,s_{i+1})=(\pm1,\pm1)$. The cases of $s_i=s_{i+1}$ are proved similarly to Lemma \ref{lem eigen gl2}.

The case of $s_i=-s_{i+1}=1$ is similar to Lemma \ref{lem 4.5}. Namely, we want to show
\begin{align*}
&\Big(1-\frac{T_i^{\bm s}y_{i-1}[-1]y_i[1]}{T_i^{\bm s}[1]y_{i-1}y_i}\, \tau\Big)\Big(1-\frac{T_{i+1}^{\bm s}y_{i}[1]y_{i+1}[-1]}{T_{i+1}^{\bm s}[-1]y_{i}y_{i+1}}\, \tau\Big)^{-1}\\
=&\ \Big(1-\frac{T_i^{\bm{s}^{[i]}}y_{i-1}[1]\tl y_i[-1]}{T_i^{\bm{s}^{[i]}}[-1]y_{i-1}\tl y_i}\, \tau\Big)^{-1}\Big(1-\frac{T_{i+1}^{\bm{s}^{[i]}}\tl y_{i}[-1]y_{i+1}[1]}{T_{i+1}^{\bm{s}^{[i]}}[1]\tl y_{i}y_{i+1}}\, \tau\Big).
\end{align*}
This equation is proved by a direct computation using Lemma \ref{eq relations diff} and \eqref{fermionic rp}. We only note that the following identities
\[
\frac{T_{i}^{\bm{s}^{[i]}}}{T_{i}^{\bm{s}^{[i]}}[-1]}\frac{T_{i+1}^{\bm s}}{T_{i+1}^{\bm s}[1]}=\frac{T_{i+1}^{\bm{s}^{[i]}}}{T_{i+1}^{\bm{s}^{[i]}}[1]}\frac{T_i^{\bm s}[2]}{T_i^{\bm s}[1]}=\prod_{\substack{k=1 \\ \la_{i}^{(k,\bm s)}+\la_{i+1}^{(k,\bm s)}\ne 0}}^p\frac{x-z_k-h}{x-z_k}
\]
are used.

The case of $s_i=-s_{i+1}=-1$ is similar.
\end{proof}
We denote the rational difference operator corresponding to a population $P$ by $\cR_P$. 

\begin{rem}
Taking the quasiclassical limit $h\to 0$, a solution $\bm{t_h}$ of BAE \eqref{eq BAE XXX} tends to a solution of BAE for the Gaudin model associated to $\gl_{m|n}$ represented by a tuple $\bm{\mathcal Y}=(\mc Y_1,\dots,\mc Y_{m+n-1})$, see Remark \ref{rem BAE limit}. Note that $\tau=e^{-h\pa_x}$, we have
\[
1-\frac{T_i^{\bm s}y_{i-1}[-s_i]y_i[s_i]}{T_i^{\bm s}[s_i]y_{i-1}y_i}\, \tau=h\left(\pa_x-s_i\Big(\ln\frac{\mathscr T_i^{\bm s}\mathcal Y_{i-1}}{\mathcal Y_i}\Big)'\right)+\mathcal O(h^2),
\]
where $\mathscr T_i^{\bm s}=\prod_{k=1}^p(x-z_k)^{\la_i^{(k,\bm s)}}$, $\mc Y_0=\mc Y_{m+n}=1$. Ignoring the terms in $\mc O(h^2)$ for each factor, one gets from $\cR^{\bm s}(\bm y)$ the rational pseudo-differential operator $R^{\bm s}(\bm {\mc Y})$ defined in \cite[equation (6.5)]{HMVY18}.\qed
\end{rem}

The transfer matrix $\cT(x)$ (associated to the vector representation) can be included in a natural commutative algebra 
$\mc B$ generated by transfer matrices associated to other finite dimensional representations of $\YglMN$, c.f. \cite{KSZ}, \cite{TZZ}. We expect that similar to the even case, the rational difference operator $\cR^{\bm{s}}(\bs y)$ encodes eigenvalues of algebra $\mc B$ acting on the Bethe vector corresponding to $\bs y$, c.f \cite{T}. Then, Theorem \ref{thm diffoper inv} would assert that formulas for eigenvalues of $\mc B$ acting on $L(\bla,\bm z)$ do not depend on a choice of $\bs y$ in the population.

Similar to Lemmas \ref{lem eigen gl2} and \ref{lem gl11 eigen}, we show that formula for eigenvalue \eqref{hamiltonian}  or \eqref{glmn eigenvalue} does not change under $\gl_{m|n}$ reproduction procedure. 

\begin{lem}
Let $\bs y=(y_1,\dots,y_{m+n-1})$ be a sequence of polynomials such that there exists a polynomial $\tl y_i$
satisfying \eqref{bosonic rp} if $s_i=s_{i+1}$ or \eqref{fermionic rp} if $s_i=-s_{i+1}$.  Then $\mc E_{\bs y}(x)=\mc E_{\bs y^{[i]}}(x)$, where
 $\bs y^{[i]}=(y_1,\dots,\tl y_i,\dots,y_{m+n-1})$.
\end{lem}
\begin{proof}
The proof is similar to proofs of Lemmas \ref{lem eigen gl2} and \ref{lem gl11 eigen}.
\end{proof}

\subsection{Example of a $\gl_{2|1}$ population}
In this section, we give an example of a population for the case of  $\gl_{2|1}$. 

Set $m=2$, $n=1$, and $p=3$. There are three parity sequences in $S_{2|1}$, namely, $\bm{s}_0=(1,1,-1)$, $\bm{s}_1=(1,-1,1)$, and $\bm{s}_2=(-1,1,1)$. 

Let $\bm{\lambda}=(\lambda^{(1)},\lambda^{(2)},\lambda^{(3)})$, where  $\lambda^{(i)}=(1,1,0)$, for $i=1,2,3$, in standard parity sequence $\bm s_0$. Let $\bm l=(0,0)$ and $\bm{y}=(y_1,y_2)=(1,1)$. We also set $h=1$.

Let $\bm{z}=(0,\sqrt{2},-\sqrt{2})$. Our choice of $\bm z$ is such that $z_i-z_j\notin h\Z$ for $i\ne j$. We have $\bm{T}=\bm{T}^{\bm{s}_0}=(x^3+3x^2+x-1,x^3+3x^2+x-1,1)$.We consider the population $P_{(1,1)}$ of solutions of the BAE associated to $\bm s_0$, $\bm z$, $\bla$, originated at $\bm y$.

\begin{enumerate}
\item Applying bosonic reproduction procedure in the first direction to $\bm{y}$, we have $\bm{s}_0^{[1]}=\bm{s}_0$, $\bm{T}^{\bs s_0}=\bm{T}$, and $\bm{y}^{[1]}_c=(y_1^{[1]},y_2^{[1]})=(x-c,1)$, where $c\in\C\bP^1$. Note that $\bm{y}_{\infty}^{[1]}=(1,1)=\bm{y}$.
\item We then apply fermionic reproduction procedure in the second direction to $\bm{y}_c^{[1]}$. We have $(\bm{s}_0)^{[2]}=\bm{s}_1$ and  $\bs T^{\bs s_1}=(x^3+3x^2+x-1,x^3-3x^2+x+1,1)$. We have 
\[
(\bm{y}_{c}^{[1]})^{[2]}=(x-c,4x^3-(6+3c)x^2+3cx+c+1).
\]
\item Finally, apply fermionic reproduction procedure in the first direction to $(\bm{y}_c^{[1]})^{[2]}$. We have  $(\bm{s}_1)^{[1]}=\bm{s}_2$ and $\bs T^{\bs s_2}=\left((x-1)(x-2)(x^2-2x-1)(x^2-4x+2),1,1\right)$. We have
\[
((\bm{y}_c^{[1]})^{[2]})^{[1]}=\big(6(x-1)^4-9(x-1)^2+1,4x^3-(6+3c)x^2+3cx+c+1\big).
\]
\end{enumerate}

It is easy to check that all further reproduction procedures cannot create a new pair of polynomials. Therefore the $\gl_{2|1}$ population $P_{(1,1)}$ is the union of three $\C\bP^1$, $P_{(1,1)}^{\bm{s}_0}=\{(x-c,1)\;|\;c\in\C\bP^1\}$, $P_{(1,1)}^{\bm{s}_1}=\{(x-c,4x^3-(6+3c)x^2+3cx+c+1)\;|\;c\in\C \bP^1\}$, and $P_{(1,1)}^{\bm{s}_2}=\{(6(x-1)^4-9(x-1)^2+1 ,4x^3-(6+3c)x^2+3cx+c+1)\;|\;c\in\C\bP^1\}$.

\section{Populations and superflag varieties}\label{sec pop flag variety}
In this section, we show that $\gl_{m|n}$ populations associated to typical $\bla$ are isomorphic to the variety of the full superflags.

\subsection{Discrete exponents and dominants}
Following \cite{HMVY18}, we introduce the following partial ordering on the set of partitions with $r$ parts. Let $\bm a=(a_1\lle a_2\lle \dots\lle a_r)$ and $\bm b=(b_1\lle b_2\lle \dots\lle b_r)$, $a_i,b_i\in \Z_{\gge 0}$, be two partitions with $r$ parts. If $b_i\gge a_i$ for all $i=1,\dots,r$, we say that $\bm b$ \emph{dominates} $\bm a$.

For a partition $\bm a$ with $r$ parts, we call the smallest partition  with $r$ distinct parts that dominates $\bm a$ the \emph{dominant} of $\bm a$ and denote it by $\bar{\bm a}=(\bar{a}_1<\bar{a}_2<\dots<\bar{a}_r)$. Namely, the partition $\bar{\bm a}$ is such that $\bar{\bm a}$ dominates $\bm a$ and if a partition $\bm a'$ with $r$ distinct parts dominates $\bm a$ then $\bm a'$ dominates $\bar{\bm a}$. The partition $\bar{\bm a}$ is unique.

We identify a set of non-negative integers with a partition by rearranging their elements into weakly increasing order.

This definition is motivated by the relation of exponents for a sum of spaces of functions to exponents of the summands. 
We describe this phenomenon for the discrete exponents of spaces of functions.

Let $V$ be an $r$-dimensional space of functions. Let $z\in \C$ be such that all functions in $V$ are well-defined at $z-h\Z$. Then there exists a partition with $r$ distinct parts $\bm c=(c_1<\cdots<c_r)$ and a basis of $\{v_1,\dotsm,v_r\}$ of $V$ such that for $i=1,\dots,r$, we have $v_i(z-jh)=0$ for $j=1,\dots,c_i$ and $v_i(z-(c_i+1)h)\ne 0$. This sequence of integers is defined uniquely and will be called the \emph{sequence of discrete exponents of $V$ at $z$}. We denote the set $\bm c$ by $\bm E_z(V)$.

Let $V_1,\dots,V_k$ be spaces of functions such that the sum $V=\sum_{i=1}^kV_i$ is a direct sum. Let $\bm a_z=\sqcup_{i=1}^k \bm E_z(V_i)$, then $\bm E_z(V)$ dominates $\bar{\bm a}_z$. Moreover, for generic spaces of functions $V_i$, we have the equality $\bm E_z(V)=\bar{\bm a}_z$.

\subsection{Space  of rational functions associated to a solution of BAE}
\label{sec 6.2}
Let $\bm\lambda=(\lambda^{(1)},\dots,\lambda^{(p)})$ be a sequence of polynomial $\gl_{m|n}$ weights. Let $\bs z=(z_1,\dots,z_p)$ be an $h$-generic sequence of complex numbers.

Let $\bm y=(y_1,\dots,y_{m+n-1})$ represent a solution of the BAE associated to $\bs \la,\bs z$, and the standard parity sequence $\bs s_0$. Suppose further that $\bs y$ is generic with respect to $\bs \la,\bs z,\bm s_0$. Recall the rational difference operator $\cR^{\bm s_0}(\bs y)=\cD_{\bar 0}(\bs y)\cD_{\bar 1}^{-1}(\bs y)$ associated to the population $P_{(\bm y,\bm s_0)}$ generated by $\bm y$, see \eqref{eq rational diff oper glmn}. Let $V_{\bm y}=\ker \cD_{\bar 0}(\bs y)$ and $U_{\bm y}=\ker \cD_{\bar 1}(\bs y)$.

Note that the sequence $(y_1,\dots,y_{m-1})$ represents a solution of the BAE associated to the Lie algebra $\gl_m$. It follows from \cite{MV03} that one can generate a $\gl_m$ population starting from $(y_1,\dots,y_{m-1})$ using bosonic reproduction procedures. Moreover, the corresponding difference operator to this population is given by $y_m \cdot \cD_{\bar 0}(\bs y)\cdot (y_m)^{-1}$. Therefore, by \cite[Proposition 4.7]{MV03}, the space $y_m\cdot V_{\bm y}$ is an $m$-dimensional space of polynomials. Similarly, since $(y_{m+1},\dots,y_{m+n-1})$ represents a solution of the BAE associated to the Lie algebra $\gl_n$, the space $T_{m+1}[-1]y_m\cdot U_{\bm y}$ is an $n$-dimensional space of polynomials. In particular, $V_{\bm y}$ and $U_{\bm y}$ are spaces of rational functions.

In the remainder of Section \ref{sec pop flag variety}, we impose the condition that $y_m(z_i+kh)\ne 0$ for $i=1,\dots,p$ and $k\in\Z$. 

Since $\bm z$ is $h$-generic and $y_m(z_i+kh)\ne 0$ for $1\lle i\lle p$ and $k\in\Z$, it follows from \cite[Corollary 7.5]{MTV2} that the sequence of discrete exponents $\bm E_{z_i}(y_m\cdot V_{\bm y})$ is given by
\[
\big(\la^{(i)}_{m}<\la^{(i)}_{m-1}+1<\dots<\la^{(i)}_{m-k+1}+k-1<\dots<\la^{(i)}_{1}+m-1\big).
\]
Therefore the sequence of discrete exponents $\bm E_{z_i+\la^{(i)}_{m+1}h}(T_{m+1}[-1]y_m\cdot V_{\bm y})$ is given by
\[
\big(\la^{(i)}_{m}+\la^{(i)}_{m+1}<\la^{(i)}_{m-1}+\la^{(i)}_{m+1}+1<\dots<\la^{(i)}_{m-k+1}+\la^{(i)}_{m+1}+k-1<\dots<\la^{(i)}_{1}+\la^{(i)}_{m+1}+m-1\big).
\]
Similarly, the sequence of discrete exponents $\bm E_{z_i+\la^{(i)}_{m+1}h}(T_{m+1}[-1]y_m\cdot U_{\bm y})$ is given by
\[
\big(0<\la^{(i)}_{m+1}-\la^{(i)}_{m+2}+1<\dots<\la^{(i)}_{m+1}-\la^{(i)}_{m+k}+k-1<\dots<\la^{(i)}_{m+1}-\la^{(i)}_{m+n}+n-1\big).
\]

\begin{lem}\label{lem empty cap}
If $\bla$ is typical, then $V_{\bm y}\cap U_{\bm y}=0$.
\end{lem}
\begin{proof}
Since $\bla$ is typical, there exists some $i_0\in\{1,\dots,p\}$ such that $\la^{(i_0)}_{m}\gge n$. Therefore the largest discrete exponent of $T_{m+1}[-1]y_m\cdot U_{\bm y}$ at $z_{i_0}+\la^{(i_0)}_{m+1}h$ is strictly less than the smallest discrete exponent of $T_{m+1}[-1]y_m\cdot V_{\bm y}$ at $z_{i_0}+\la^{(i_0)}_{m+1}h$, namely,
\[
\la^{(i_0)}_{m+1}-\la^{(i_0)}_{m+n}+n-1<n+\la_{m+1}^{(i_0)}\lle \la^{(i_0)}_{m}+\la^{(i_0)}_{m+1}.
\]
Therefore, by the definition of discrete exponents, we have $(T_{m+1}[-1]y_m\cdot U_{\bm y})\cap (T_{m+1}[-1]y_m\cdot V_{\bm y})=0$, which completes the proof.
\end{proof}

Therefore, by Proposition \ref{prop rdp}, the operator $\cR^{\bm s_0}(\bs y)$ is an $(m|n)$-rational difference operator.

\begin{rem}	
If $\bla$ is not typical, then the intersection $V_{\bm y}\cap U_{\bm y}$ may be non-trivial. For example, consider the tensor product of the vector representations, namely $L(\bla)=(\C^{m|n})^{\otimes p}$, and the sequence of polynomials $\bm y=(1,\dots,1)$. Then we have $T_1(x)=(x-z_1+h)\cdots(x-z_p+h)$ and $T_i(x)=1$ for $i=2,\dots,m+n$. Therefore for the rational difference operator $\cR^{\bm s_0}(\bs y)=\cD_{\bar 0}(\bs y)\cD_{\bar 1}^{-1}(\bs y)$, we have
\[
\cD_{\bar 0}(\bs y)=\Big(1-\frac{(x-z_1+h)\cdots(x-z_p+h)}{(x-z_1)\cdots(x-z_p)}\tau\Big)(1-\tau)^{m-1},\qquad \cD_{\bar 1}(\bs y)=(1-\tau)^{n}.\qedd
\]
\end{rem}

Fix $a\in \{0,1,\dots,m\}$ and $b\in\{0,1,\dots,n\}$. For each $1\lle i\lle p$, set $$\bm A_i=\big(\la^{(i)}_{m}+\la^{(i)}_{m+1}<\la^{(i)}_{m-1}+\la^{(i)}_{m+1}+1<\dots<\la^{(i)}_{m-a+1}+\la^{(i)}_{m+1}+a-1\big),$$ $$\bm B_i=\big(0<\la^{(i)}_{m+1}-\la^{(i)}_{m+2}+1<\dots<\la^{(i)}_{m+1}-\la^{(i)}_{m+b}+b-1\big).$$

\begin{lem}\label{lem dom}
If $b\lle \la_m^{(i)}$, then the dominant of $\bm A_i\sqcup \bm B_i$ is given by\\
\scalebox{0.97}{\parbox{\linewidth}{%
\begin{align*}
(0<\la^{(i)}_{m+1}-\la^{(i)}_{m+2}+1<\dots<\la^{(i)}_{m+1}-\la^{(i)}_{m+b}+b-1<\la^{(i)}_{m}+\la^{(i)}_{m+1}<\dots<\la^{(i)}_{m-a+1}+\la^{(i)}_{m+1}+a-1).
\end{align*}
}}\\
If $\la_{m-j+1}^{(i)}<b\lle \la_{m-j}^{(i)}$ for some $1\lle j\lle a-1$, then the dominant of $\bm A_i\sqcup \bm B_i$ is given by 
\begin{align*}
(0<\la^{(i)}_{m+1}-\la^{(i)}_{m+2}+1<\dots<\la^{(i)}_{m+1}-\la^{(i)}_{m+b}+b-1<\la^{(i)}_{m+1}+b<\la^{(i)}_{m+1}+b+1<\dots<\\\la^{(i)}_{m+1}+b+j-1< \la^{(i)}_{m-j}+\la^{(i)}_{m+1}+j<\dots<\la^{(i)}_{m-a+1}+\la^{(i)}_{m+1}+a-1).
\end{align*}
If $\la_{m-a+1}^{(i)}<b$, then the dominant of $\bm A_i\sqcup \bm B_i$ is given by\\
\scalebox{0.915}{\parbox{\linewidth}{%
\begin{align*}
(0<\la^{(i)}_{m+1}-\la^{(i)}_{m+2}+1<\dots<\la^{(i)}_{m+1}-\la^{(i)}_{m+b}+b-1<\la^{(i)}_{m+1}+b<\la^{(i)}_{m+1}+b+1<\dots<\la^{(i)}_{m+1}+b+a-1).
\end{align*}
}}
\end{lem}
\begin{proof}
If $b\lle \la_m^{(i)}$, the statement is clear. If $\la_{m-j+1}^{(i)}<b\lle \la_{m-j}^{(i)}$ for some $1\lle j\lle a-1$. Let $\la_m^{(i)}=\ell$. Since $\la^{(i)}$ is a polynomial $\glMN$ weight, we have $\la_{m+\ell+k}^{(i)}=0$ for $k=1,\dots,b-\ell$. In particular, the last $b-\ell$ numbers in $\bm B_i$ are consecutive integers from $\la_{m+1}^{(i)}+\ell$ to $\la_{m+1}^{(i)}+b-1$. Adding $\la^{(i)}_{m}+\la^{(i)}_{m+1}$ into $\bm B_i$, the dominant of the new set is obtained by changing $\la^{(i)}_{m}+\la^{(i)}_{m+1}$ to $\la_{m+1}^{(i)}+b$. We add the numbers of $\bm A_i$ one by one (from left to right) into $\bm B_i$. Inductively, adding $\la^{(i)}_{m+1}+\la^{(i)}_{m-k+1}+k-1$, if $\la^{(i)}_{m-k+1}<b$, then the dominant is obtained by changing $\la^{(i)}_{m+1}+\la^{(i)}_{m-k+1}+k-1$ to $\la^{(i)}_{m+1}+b+k-1$. Therefore the lemma follows.
\end{proof}

\subsection{Polynomials $\pi_{a,b}$}
Let $\bm{s}\in S_{m|n}$ be a parity sequence. Let $\bm{\lambda}=(\lambda^{(1)},\dots,\lambda^{(p)})$ be a sequence of polynomial $\gl_{m|n}$ weights. Let $\bm{z}=(z_1,\dots,z_p)$ be an $h$-generic sequence of complex numbers. Let $\bm{T}^{\bm{s}}$ be a sequence of polynomials associated to $\bm{s}$, $\bm{\lambda}$, and $\bm{z}$, see \eqref{eq functions T}. We set $T_i=T_i^{\bs s_0}$ the polynomials corresponding to the standard parity $\bm s_0$. 

Define polynomials $\pi_{a,b}^{\bla,\bm z}$ by 
\be
\pi_{a,b}^{\bla,\bm z}(x)=\prod_{k=1}^p\prod_{i=1}^a\prod_{j=1}^{\min\{b,\la_{m-i+1}^{(k)}\}}(x-z_k+(i+j-a-b-1)h).
\ee
We often abbreviate $\pi_{a,b}^{\bla,\bm z}$ to $\pi_{a,b}$. 

The polynomials $T_i^{\bs s}$ can be expressed in terms of $T_i$ and $\pi_{a,b}$.
Recall that we have
\be
\bs s_i^+=\begin{cases} m-\sigma_{\bs s}(i), & \mbox{if } s_i=1,\\  \sigma_{\bs s}(i)-i,  & \mbox{if } s_i=-1, \end{cases}
\qquad 
\bs s_i^-=\begin{cases}i-\sigma_{\bs s}(i), & \mbox{if } s_i=1,\\  \sigma_{\bs s}(i)-m-1,  & \mbox{if } s_i=-1. \end{cases}
\ee
\begin{thm}\label{thm relation between T}
We have
\begin{align*}
T^{\bs s}_i= \ 
T_{\sigma_{\bs s}(i)}[\bm s_i^-]\ \frac{\pi_{\bs s^+_i,\bs s^-_i}}{\pi_{\bs s^+_i+1,\bs s^-_i}[-1]},\; {\rm if }\ s_i=1;\qquad
T^{\bs s}_i=\ T_{\sigma_{\bs s}(i)}[\bm s_i^+]\ \frac{\pi_{\bs s^+_i,\bs s^-_i+1}}{\pi_{\bs s^+_i,\bs s^-_i}[1]},\;  {\rm if }\ s_i=-1.
\end{align*}
\end{thm}
\begin{proof}
It is not hard to see that
\[
\la_i^{(k,\bm s)}=\begin{cases}
\la_{\sigma_{\bm s}(i)}^{(k)}-\min\big\{\bm s_i^-,\la_{\sigma_{\bm s}(i)}^{(k)}\big\}, & \text{ if }s_i=1,\\
\la_{\sigma_{\bm s}(i)}^{(k)}+\#\{j~|~\la_{m-j+1}^{(k)}>\bm s_i^-,j=1,2,\dots,\bm s_i^+\}, & \text{ if }s_i=-1.
\end{cases}
\]
The theorem follows from a direct computation.
\end{proof}

Note that polynomials $\pi_{a,b}$ are discrete versions of $\pi_{a,b}$ in \cite[equation (7.1)]{HMVY18}, even though our definition here is more explicit.
In particular, Theorem \ref{thm relation between T} is the counterpart of \cite[Theorem 7.2]{HMVY18}.

The polynomial $\pi_{a,b}$ is related to the dominants of $\bm A_i\sqcup \bm B_i$ for all $1\lle i\lle p$. Write the dominant $\overline{\bm A_i\sqcup \bm B_i}$ of $\bm A_i\sqcup \bm B_i$ as
\[
0=c_{a+b}^{(i)}<c_{a+b-1}^{(i)}+1<\dots<c_{a+b-j}^{(i)}+j<\dots<c_1^{(i)}+a+b-1,
\]
where $c_j^{(i)}$ are computed explicitly from Lemma \ref{lem dom}. Let $\tl z_i=z_i+\la_{m+1}^{(i)}h$ and set
\beq\label{eq poly T divisible}
\mathscr T_i(x)=\prod_{k=1}^p\prod_{j=1}^{c_i^{(k)}}(x-\tl z_k+jh).
\eeq
\begin{prop}\label{prop pi}
We have
\[
\pi_{a,b}\prod_{j=1}^{a}\mathscr T_j[j]=\prod_{i=1}^a\big(T_{m-a+i}[b+i]T_{m+1}[i-1]\big).
\]
\end{prop}
\begin{proof}
The lemma is obtained from Lemma \ref{lem dom} by a direct computation.
\end{proof}

\subsection{Generating map}Recall the notation from the beginning of Section \ref{sec 6.2}, where $V_{\bm y}=\ker \cD_{\bar 0}(\bs y)$ and $U_{\bm y}=\ker \cD_{\bar 1}(\bs y)$.

For $a\in\{0,1,\dots,m\}$, $b\in\{0,1,\dots,n\}$, $v_1,\dots, v_a\in V_{\bm y}$, $u_1,\dots, u_b\in U_{\bm y}$, we define the  function
\begin{align*}
y_{a,b}=&\ \Wr(v_1,\dots,v_{a},u_1,\dots,u_b)[1]\pi_{a,b}  y_m[a+b]\frac{T_{m+1}[a+b-1]\cdots T_{m+b}[a] }{T_{m}[a+b]\cdots T_{m-a+1}[b+1]}\,.
\end{align*}

We impose the technical condition that $y_m$ has only simple roots and is relatively prime to $y_m[k]$ for all non-zero integers $k$.

\begin{prop}\label{prop y polynomial}
The function $y_{a,b}$ is a polynomial. 
\end{prop}
\begin{proof}
	This proposition is proved in Section \ref{sec proof}.
\end{proof}

In the following ,we assume that $\bla$ is typical. Set $W_{\bm y}=V_{\bm y}\oplus U_{\bm y}$. Given a parity sequence $\bs s$ and a full superflag $\sF\in\sF^{\bs s}(W_{\bm y})$ generated by a homogeneous basis $\{w_1,\dots,w_{m+n}\}$, we define polynomials $y_i(\sF)$, $i=1,\dots, m+n-1$, by the formula 
$$
y_i(\sF)=\begin{cases} y_{\bs s_i^+,\bs s_i^-}, \ &{\rm if} \ s_i=1, \\  y_{\bs s_i^+,\bs s_i^-+1}, \ &{\rm if} \ s_i=-1,\end{cases}
$$
where we choose $\{v_1,\dots,v_m\}$ and $\{u_1,\dots,u_n\}$ such that the basis $\{w_1,\dots,w_{m+n}\}$ is associated to $\{v_1,\dots,v_m\}$, $\{u_1,\dots,u_n\}$, and $\bm s$, see Section \ref{sec 2.3}.

Define the {\it generating map} by
\[\beta^{\bs s}:\sF^{\bs s}(W_{\bm y})\to \big(\mathbb{P}(\C[x])\big)^{m+n-1}, \quad \sF\mapsto \bs y(\sF)=(y_1(\sF),\dots,y_{m+n-1}(\sF)).\]
The following theorem is our main result of this section.

\begin{thm}\label{thm bijection 3 objs}
For any superflag $\sF\in\sF^{\bm s}(W_{\bm y})$, we have $\beta^{\bm s}(\sF)\in P_{(\bm y,\bm s_0)}^{\bm s}$. Moreover, the generating map $\beta^{\bm s}: \sF^{\bm s}(W_{\bm y})\to P_{(\bm y,\bm s_0)}^{\bm s}$ is a bijection and the complete factorization $\varpi^{\bm s}(\sF)$ of $\cR^{\bm s_0}(\bm y)$ given by \eqref{eq wronski coeff} coincides with $\cR^{\bm s}(\beta^{\bm s}(\sF))$ given by \eqref{eq rational diff oper glmn}.
\end{thm}
\begin{proof}
Note that the even case of this theorem is proved in \cite[Theorem 4.16]{MV03}. Due to Theorem \ref{thm relation between T} and Proposition \ref{prop y polynomial}, the proof is parallel to that of \cite[Theorem 7.9]{HMVY18}.
\end{proof}

This theorem does not rely on the technical condition imposed above Proposition \ref{prop y polynomial}, see Remark \ref{rem typical poly}.

\subsection{Proof of Proposition \ref{prop y polynomial}}
\label{sec proof}
We prepare several lemmas which will be used in the proof.
\begin{lem}\label{lem skip exp}
For any $v\in V_{\bm y},u\in U_{\bm y}$, the function $T_{m+1}y_m[1]\Wr(v,u)$ is a polynomial. In particular, if $v\in V_{\bm y},u\in U_{\bm y}$ are not regular at $z$, then there exists a $c\in\C$ such that $(u+cv)(z-h)=0$.
\end{lem}
\begin{proof}
The case of $\gl_{1|1}$ is clear. Now we assume that either $m\gge 2$ or $n\gge 2$.

If the fermionic reproduction in the $m$-th direction is not applicable, then we can slightly change $y_{m-1}$ or $y_{m+1}$ using bosonic reproduction procedure such that the fermionic reproduction in the $m$-th direction can be applied to the new tuple of polynomials $\tl{\bm y}$. Therefore we can assume that the fermionic reproduction in the $m$-th direction is applicable to $\bm y$ at the beginning.

It follows from \eqref{eq wronski coeff} and Theorem \ref{thm repro pro glmn} that
\[
T_{m+1}y_m[1]\Wr(v,u)=T_{m+1}^{\bm s_0^{[m]}}\tl y_{m}[-1].
\]
Here $\tl y_m$ depends on $u$ and $v$. 

Initially, we have $v(\bm y)=T_{m}y_{m-1}[-1]/y_m$ and $u(\bm y)=y_{m+1}[-1]/(T_{m+1}[-1]y_m)$. Generic $u$ and $v$ can be obtained from $\bm y$ using only bosonic reproduction procedures. Moreover, the polynomial $y_m$ never changes. Note that, by Theorem \ref{thm repro pro glmn}, $\tl y_m$ is a polynomial for generic $u$ and $v$. Therefore the first part of the lemma follows.

Recall that $y_m$ has only simple zeros and $y_m$ is relatively prime to $y_m[1]$. In addition, none of zeros of $y_m$ belongs to the sets $z_k+h\Z$, $k=1,\dots,p$. If $v\in V_{\bm y},u\in U_{\bm y}$ are not regular at $z$, then $z$ is a root of $y_m$. Moreover, $v$ and $u$ have simple pole at $x=z$. The second statement follows directly from the first statement.
\end{proof}

Suppose $V$ is an $r$-dimensional space of polynomials with the sequence of discrete exponents at $z$ given by $c_r<c_{r-1}+1<\dots<c_{r-i}+i<\dots<c_1+r-1$. Let $\mathscr T_i(x)=(x-z+h)\cdots(x-z+c_{i}h)$, $i=1,\dots,r$. 

The following lemma is well-known, see e.g. \cite[Theorem 3.3]{MTV3}.
\begin{lem}\label{lem divisible}
Let $f_1,\dots,f_i\in V$, then $\Wr(f_1,\dots,f_i)$ is divisible by $\mathscr \prod_{j=1}^{i}\mathscr T_{r+1-j}[i-j]$.\qed
\end{lem}
\begin{proof}[Proof of Proposition \ref{prop y polynomial}]
	Clearly, we only need to consider the case when $v_1,\dots,v_{a},u_1,\dots,u_b$ are linearly independent. The rational function $y_{a,b}$ can only have poles at $z_i+h\Z$, $1\lle i\lle p$, and at zeros of the product of polynomials $\prod_{j=1}^{a+b}y_{m}[j]$.
	
	Denote by $W_{a,b}$ the space of polynomials spanned by $\tl v_j:=T_{m+1}[-1]y_mv_j,\tl u_k:=T_{m+1}[-1]y_mu_k$, $1\lle j\lle a$ and $1\lle k\lle b$, then $\bm E_{\tl z_i}(W_{a,b})$ dominates $\overline{\bm A_i\sqcup \bm B_i}$, where $\tl z_i=z_i+\la_{m+1}^{(i)}h$. Therefore it follows from Lemma \ref{lem divisible} that $\Wr(\tl v_1,\dots,\tl v_a,\tl u_1,\dots,\tl u_b)$ is divisible by $\prod_{j=1}^{a+b}\mathscr T_{j}[j-1]$, where $\mathscr T_{j}$ are defined in \eqref{eq poly T divisible}. It follows from Proposition \ref{prop pi} that the function $y_{a,b}$ is regular at $z_i+h\Z$, $1\lle i\lle p$.
	
	Write $y_m=\prod_{i=1}^r(x-z'_i+h)$, then by assumption $z'_i-z'_j\notin h\Z$ for $1\lle i<j\lle r$. It follows from \cite[Corollary 7.5]{MTV2} that $\bm E_{z'_i}(\mathrm{span}\langle \tl v_1,\dots,\tl v_a\rangle)$ dominates the partition $(0<2<3<\dots<a)$ with $a$ parts and $\bm E_{z'_i}(\mathrm{span}\langle \tl u_1,\dots,\tl u_b\rangle)$ dominates the partition $(0<2<3<\dots<b)$ with $b$ parts. Therefore it follows from Lemma \ref{lem skip exp} that $\bm E_{z'_i}(W_{a,b})$ dominates the partition $(0<2<3<\dots<a+b)$ with $a+b$ parts. Hence, by Lemma \ref{lem divisible}, $\Wr(\tl v_1,\dots,\tl v_a,\tl u_1,\dots,\tl u_b)$ is divisible by $\prod_{j=2}^{a+b}y_m[j-2]$. In particular, $\Wr(v_1,\dots,v_a,u_1,\dots,u_b)y_m[a+b-1]$ is regular at zeros of the product of polynomials $\prod_{j=1}^{a+b}y_{m}[j-1]$.
\end{proof}
\begin{rem}\label{rem typical poly}
If $\bla$ is typical, the proof of Proposition \ref{prop y polynomial} can be simplified as follows. Since $\bla$ is typical, generically the reproduction procedure is applicable for all parity sequences and all directions. Therefore, it follows from Theorem \ref{thm repro pro glmn} that $y_{a,b}$ is a polynomial for generic $v_1,\dots,v_a$, $u_1,\dots,u_b$. Hence $y_{a,b}$ is a polynomial for all $v_1,\dots,v_a$, $u_1,\dots,u_b$.\qed
\end{rem}

\section{Quasi-periodic Case}
\label{sec quasi}
In this section, we generalize our results to the quasi-periodic case.
\subsection{Twisted transfer matrix and Bethe ansatz}
We follow the notation in Section \ref{sec super XXX bae}.

Let $\bm\ka=(\ka_1,\dots, \ka_{m+n})$ be a sequence of complex numbers such that $e^{h\ka_i}\ne e^{h\ka_j}$ for $1\lle i<j\lle m+n$. Let $Q_{\bka}$ be the diagonal matrix $\mathrm{diag}(e^{h\ka_1 },\dots,e^{h\ka_{m+n}})$. Define the twisted transfer matrix $\cT_{\bka}(x)$ by
\[
\cT_{\bka}(x)=\mathrm{str}(Q_{\bka}\mc L(x))=\sum_{i=1}^{m+n}(-1)^{|i|}e^{h\ka_i}\mc L_{ii}(x).
\]
It is known that the twisted transfer matrices commute, $[\mc T_{\bka}(x_1),\mc T_{\bka}(x_2)]=0$. Moreover, $\cT_{\bka}(x)$ commutes with the subalgebra $\mathrm{U}(\h)$.

The Bethe ansatz equation associated to $\bm s$, $\bm z$, $\bla$, $\bka$, and $\bm l$ is a system of algebraic equations in variables $\bm t$:
\begin{align}
e^{h(\ka_i-\ka_{i+1})}\prod_{k=1}^p\frac{t_j^{(i)}-z_k+s_i\la_i^{(k,\bm s)}h}{t_j^{(i)}-z_k+s_{i+1}\la_{i+1}^{(k,\bm s)}h}&\prod_{r=1}^{l_{i-1}}\frac{t_j^{(i)}-t_r^{(i-1)}+s_ih}{t_j^{(i)}-t_r^{(i-1)}} \nonumber\\\times &\prod_{\substack{r=1 \\ r\ne j}}^{l_{i}}\frac{t_j^{(i)}-t_r^{(i)}-s_ih}{t_j^{(i)}-t_r^{(i)}+s_{i+1}h}
\prod_{r=1}^{l_{i+1}}\frac{t_j^{(i)}-t_r^{(i+1)}}{t_j^{(i)}-t_r^{(i+1)}-s_{i+1}h}=1,\label{twisted bae}
\end{align}
where $i=1,\dots,m+n-1$, $j=1,\dots,l_i$.

After making cancellations as in \eqref{eq cancel}, we require the solutions do not make the remaining denominators in \eqref{twisted bae} vanish.

We also impose the same condition, see Section \ref{sec super XXX bae}, for variables which correspond to a simple odd root of the same color. Suppose $(\alpha_i^{\bm s},\alpha_i^{\bm s})=0$ for some $i$. Consider the BAE for $\bm t$ related to $t_j^{(i)}$ with all $t_b^{(a)}$ fixed, where $a\ne i$ and $1\lle b\lle l_a$, this equation does not depend on $j$. Let $t_0^{(i)}$ be a solution of this equation with multiplicity $r$. Then we require that the number of $j$ such that $t_j^{(i)}=t_0^{(i)}$ is at most $r$, c.f. Theorem \ref{Repro twisted}.

Suppose that $\bm{\lambda}$ is a sequence of polynomial $\glMN$ weights and $\bs t$ a solution of the BAE \eqref{twisted bae} associated to $\bs s$, $\bs z$, $\bm\lambda$, $\bka$, and $\bm{l}$. Similar to Theorem \ref{thm BR 08}, see \cite{BR2}, if the vector $w^{\bs s}(\bs t,\bs z)\in L(\bla,\bm z)$ is well-defined and non-zero, then $w^{\bs s}(\bs t,\bs z)\in L(\bla,\bm z)$ is an eigenvector of twisted transfer matrix, $\mc T_{\bka}(x)w^{\bs s}(\bs t,\bs z)=\mc E_{\bka}(x)w^{\bs s}(\bs t,\bs z)$, where the eigenvalue $\mc E_{\bka}(x)$ is given by
\be
\mc E_{\bka}(x)=\sum_{a=1}^{m+n}s_a\,e^{h\ka_a}\,\prod_{k=1}^p\frac{x-z_k+s_a\la_a^{(k,\bm s)}h}{x-z_k}\prod_{j=1}^{l_{a-1}}\frac{x-t_j^{(a-1)}+s_ah}{x-t_{j}^{(a-1)}}\prod_{j=1}^{l_a}\frac{x-t_j^{(a)}-s_ah}{x-t_{j}^{(a)}}\ .
\ee

Let $\bm y=(y_1,\dots,y_{m+n-1})$ be a sequence of polynomials representing the solution $\bm t$, then
\[
\mc E_{\bka}(x)=\mc E_{(\bm y,\bka)}(x)=\sum_{a=1}^{m+n}s_a\,e^{h\ka_a}\,\frac{T_a^{\bm s}}{T_a^{\bm s}[s_a]}\frac{y_{a-1}[-s_a]}{y_{a-1}}\frac{y_a[s_a]}{y_a}\,.
\]

\subsection{Reproduction procedure and rational difference operators}
Recall the notation given at the beginning of Section \ref{subsec glmn rep pro}.
Set $\bka^{[i]}=(\ka_1,\dots,\ka_{i+1},\ka_i,\dots,\ka_{m+n})$. 

\begin{thm}\label{Repro twisted}
	Let $\bm{y}=(y_1,\dots,y_{m+n-1})$ be a sequence of polynomials generic with respect to $\bm{s}$, $\bm{\lambda}$, and $\bm{z}$, such that $\deg y_k=l_k$, $k=1,\dots,m+n-1$.
	\begin{enumerate}
		\item The sequence $\bm{y}$ represents a solution of the BAE \eqref{twisted bae} associated to $\bm{s}$, $\bm{z}$, $\bm{\lambda}$, $\bka$, and $\bm{l}$, if and only if for each $i=1,\dots,m+n-1$, there exists a unique polynomial $\tl{y}_i$, such that
		\begin{align}
		&\Wr^{s_i}\left(y_i,e^{(\ka_i-\ka_{i+1}) x}\tl{y}_i\right)=e^{(\ka_i-\ka_{i+1}) x}T^{\bm{s}}_i\left(T^{\bm{s}}_{i+1}\right)^{-1}y_{i-1}[-s_i]y_{i+1}, &\mbox{if}\;s_i=s_{i+1},\label{twisted brp}\\
		&y_i\,\tl{y}_i[-s_i]=e^{h\ka_i }\varphi_i^{\bm s}y_{i-1}[-s_i]y_{i+1}-e^{h\ka_{i+1}}\psi_{i}^{\bm s}y_{i-1}y_{i+1}[-s_i], &\mbox{if}\;s_i\neq s_{i+1}.\label{twisted frp}
		\end{align}
		\item  If $\bm{y}^{[i]}=(y_1,\dots,\tl{y}_i,\dots,y_{m+n-1})$ 
		is  generic with respect to $\bm{s}^{[i]}$, $\bm{\lambda}$, and $\bm{z}$, then $\bm{y}^{[i]}$ represents a solution of the BAE \eqref{twisted bae} associated to $\bm{s}^{[i]}$, $\bm{\lambda}$, $\bka^{[i]}$, $\bm{z}$, and $\bm{l}^{[i]}$, where
		$\bm{l}^{[i]}=(l_1,\dots,\tl{l}_i,\dots,l_{m+n-1})$, $\tl{l}_i=\deg \tl{y}_i$.
	\end{enumerate}
\end{thm}
\begin{proof}
	For part (i), the case of \eqref{twisted brp} is proved in \cite[Theorem 7.4]{MV2}. The proofs of \eqref{twisted frp} in part (i) and part (ii) are similar to that of Theorem \ref{thm repro pro glmn}.
\end{proof}

Thanks to Theorem \ref{Repro twisted}, we define similarly the twisted bosonic and fermionic reproduction procedures in $i$-th direction, the twisted $\glMN$ population $P(\bm y,\bka)$ of solutions of the BAE associated to $\bm{s}$, $\bm{z}$, $\bm{\lambda}$, originated at $(\bm y,\bka)$. Here the reproduction procedure in $i$-th direction sends $(\bm y,\bka)$ to $(\bm y^{[i]},\bka^{[i]})$. Note that for both twisted bosonic and fermionic reproduction procedures, the sequence $\bka$ is changed to $\bka^{[i]}$.

Define a rational difference operator $\cR^{\bm{s}}(\bs y,\bka)$ over $\bK=\C(x)$, 
\begin{equation}
\cR^{\bm{s}}(\bs y,\bka)=\mathop{\overrightarrow\prod}\limits_{1\lle i\lle m+n}\Big(1-e^{h\ka_i}\,\frac{T_i^{\bm s}y_{i-1}[-s_i]y_i[s_i]}{T_i^{\bm s}[s_i]y_{i-1}y_i}\, \tau\Big)^{s_i}.
\end{equation}

\begin{thm}
	Let $P$ be a twisted $\gl_{m|n}$ population. Then the rational difference operator $\cR^{\bm{s}}(\bm y,\bka)$ does not depend on a choice of $(\bm y,\bka)$ in $P$.
\end{thm}
\begin{proof}
	The proof is similar to that of Theorem \ref{thm diffoper inv}.
\end{proof}

\begin{prop}
	Let $\bs y=(y_1,\dots,y_{m+n-1})$ be a sequence of polynomials such that there exists a sequence of polynomials $\bs y^{[i]}=(y_1,\dots,\tl y_i,\dots,y_{m+n-1})$
	satisfying \eqref{twisted brp} if $s_i=s_{i+1}$ or \eqref{twisted frp} if $s_i=-s_{i+1}$.  Then $\mc E_{(\bm y,\bka)}(x)=\mc E_{(\bm y^{[i]},\bka^{[i]})}(x)$.
\end{prop}
\begin{proof}
	The proof is similar to proofs of Lemmas \ref{lem eigen gl2} and \ref{lem gl11 eigen}.
\end{proof}

Let $\sigma_i$ be the permutation $(i,i+1)$ in the symmetric group $\mathfrak S_{m+n}$. There is a natural action of $\mathfrak S_{m+n}$ on the set of sequences of $m+n$ complex numbers. Namely, for a sequence $\bka$, we have $\sigma_i\bka=\bka^{[i]}$.

\begin{thm}\label{thm bijection weyl}
	The map $P(\bm y,\bka)\to \mathfrak S_{m+n}\bka$ given by $(\tl{\bm y},\tl{\bka})\mapsto \tl{\bka}$ is a bijection between the twisted population $P(\bm y,\bka)$ and the orbit of $\bka$ under the action of symmetric group $\mathfrak S_{m+n}$. In particular, it gives a bijection between the twisted population $P(\bm y,\bka)$ and the symmetric group $\mathfrak S_{m+n}$.
\end{thm}
\begin{proof}
	The proof is similar to that of \cite[Corollary 4.12]{MV2}.
\end{proof}

\appendix
\section{The Bethe ansatz for $\mathrm Y(\gl_{1|1})$}
\label{sec gl11 eg conj}
In this section, we give the basics of Bethe ansatz for $\gl_{1|1}$ XXX model (supersymmetric spin chains associated to $\gl_{1|1}$). We follow the notation of Section \ref{sec repro pro gl11}. We also set $h=1$.

\subsection{Super Yangian $\mathrm Y(\gl_{1|1})$ and its representations}
Recall that for $\mathrm Y(\gl_{1|1})$ we have
\be
[\mc L_{ii}(x_1),\mc L_{ii}(x_2)]=0,\quad \mc L_{ij}(x_1)\mc L_{ij}(x_2)=\frac{x_1-x_2-(-1)^{|i|}}{x_2-x_1-(-1)^{|i|}}\mc L_{ij}(x_2)\mc L_{ij}(x_1),
\ee
\be
\mc L_{kk}(x_1)\mc L_{ij}(x_2)=\frac{x_1-x_2-(-1)^{|i|}}{x_1-x_2}\mc L_{ij}(x_2)\mc L_{kk}(x_1)+\frac{(-1)^{|i|}}{x_1-x_2}\mc L_{ij}(x_1)\mc L_{kk}(x_2),
\ee
where $i\ne j$ and $i,j,k\in\{1,2\}$.

In what follows we work with the standard parity sequence $\bm s_0$.

The description of finite dimensional irreducible representations of $\mathrm Y(\gl_{1|1})$ is well known. 

Let $\la=(\la_1,\la_2)$ be a $\gl_{1|1}$ weight, we say that $\la$ is \textit{non-degenerate} if $\la_1+\la_2\ne 0$. Clearly, $L_{\la}$ is two-dimensional if $\la$ is non-degenerate and one-dimensional otherwise. Let $\bla=(\la^{(1)},\dots,\la^{(p)})$ be a sequence of non-degenerate $\gl_{1|1}$ weights, $\bm z$ a sequence of complex numbers. Let $\lambda^{(k)}=(a_k,b_k)$, $a_k,b_k\in \C$,
\[
a=\sum_{k=1}^pa_k,\quad b=\sum_{k=1}^pb_k,\quad \varphi(x)=\prod_{k=1}^p(x-z_k+a_k),\quad \psi(x)=\prod_{k=1}^p(x-z_k-b_k).
\]

\begin{thm}[\cite{Z}]\label{thm fd gl11}
	Every finite dimensional irreducible representation of $\mathrm Y(\gl_{1|1})$ is a tensor product of evaluation $\mathrm Y(\gl_{1|1})$-modules up to twisting by a one-dimensional $\mathrm Y(\gl_{1|1})$-module. Moreover, $L(\bla,\bm z)$ is irreducible if and only if $\varphi(x)$ and $\psi(x)$ are relatively prime.
\end{thm}

Clearly, the $\mathrm Y(\gl_{1|1})$-module $L(\bla,\bm z)$ is irreducible if and only if $z_i-z_j-a_i-b_j\ne 0$ for all $i\ne j$. Moreover, it satisfies the binary property. Namely, $L(\bla,\bm z)$ is irreducible if and only if $L_{\la^{(i)}}(z_i)\otimes L_{\la^{(j)}}(z_j)$ is irreducible for all $1\lle i<j\lle p$. Furthermore, every finite dimensional irreducible representation of $\mathrm Y(\gl_{1|1})$ has dimension $2^r$ for some non-negative integer $r$. 

\medskip
Let $v_1^{(k)}$ be the highest weight vector of $L_{\la^{(k)}}$ with respect to the standard root system, and $v_2^{(k)}=e_{21}v_1^{(k)}$. Then $v_1^{(k)}$, $v_2^{(k)}$ is a basis of $L_{\la^{(k)}}$. We use the shorthand notation $|0\rangle$ for $v_1^{(1)}\otimes\cdots\otimes v_1^{(p)}$. 

Let $E_{ij}$, $i,j=1,2$, be the linear operator in $\End(L_{\la^{(k)}})$ of parity $|i|+|j|$ such that $E_{ij}v_{r}^{(k)}=\delta_{jr}v_i^{(k)}$ for $r=1,2$.

The R-matrix $R(x)\in \End(L_{\la^{(i)}})\otimes \End(L_{\la^{(j)}})$ is given by
\begin{align*}
R(x)=&\ E_{11}\otimes E_{11}-\frac{b_i+a_j+x}{a_i+b_j-x}E_{22}\otimes E_{22}+\frac{b_j-b_i-x}{a_i+b_j-x}E_{11}\otimes E_{22}\\
& +\frac{a_i-a_j-x}{a_i+b_j-x}E_{22}\otimes E_{11}-\frac{a_i+b_i}{a_i+b_j-x}E_{12}\otimes E_{21}+\frac{a_j+b_j}{a_i+b_j-x}E_{21}\otimes E_{12}.
\end{align*}
Clearly, $L_{\la^{(i)}}(z_i)\otimes L_{\la^{(j)}}(z_j)$ is irreducible if and only if $R(z_i-z_j)$ is well-defined and invertible.

Define an anti-automorphism $\iota:\rY(\gl_{1|1})\to \rY(\gl_{1|1})$ by the rule, $\iota(\mc L_{ij}(x))=(-1)^{|i||j|+|i|}\mc L_{ji}(x)$, $i,j=1$. One has $\iota(X_1X_2)=(-1)^{|X_1||X_2|}\iota(X_2)\iota(X_1)$ for $X_1,X_2\in \mathrm Y(\gl_{1|1})$. Recall that $\cT(x)=\mc L_{11}(x)-\mc L_{22}(x)$, therefore $\iota(\cT(x))=\cT(x)$.

The \emph{Shapovalov form} $B_{\la^{(i)}}$ on $L_{\la^{(i)}}$ is a bilinear form such that $$B_{\la^{(i)}}(e_{ij}w_1,w_2)=(-1)^{(|i|+|j|)|w_1|}B_{\la^{(i)}}(w_1,(-1)^{|i||j|+|i|}e_{ji}w_2),$$ for all $i,j$ and $w_1,w_2\in L_{\la^{(i)}}$,  and $B_{\la^{(i)}}(v^{(i)}_1,v^{(i)}_1)=1$. Explicitly, it is given by
\[
B_{\la^{(i)}}(v^{(i)}_1,v^{(i)}_1)=1,\quad B_{\la^{(i)}}(v^{(i)}_1,v^{(i)}_2)=B_{\la^{(i)}}(v^{(i)}_2,v^{(i)}_1)=0,\quad B_{\la^{(i)}}(v^{(i)}_2,v^{(i)}_2)=-(a_i+b_i).
\]
The Shapovalov forms $B_{\la^{(i)}}$ on $L_{\la^{(i)}}$ induce a bilinear form $B_{\bla}=\bigotimes_{k=1}^pB_{\la^{(k)}}$ (following the usual sign convention) on $L(\bla)$.

Let $R_{\bla,\bm z}\in \End(L(\bla))$ be the product of R-matrices,
\[
R_{\bla,\bm z}=\mathop{\overrightarrow\prod}\limits_{1\lle i\lle p}\,\mathop{\overrightarrow\prod}\limits_{i<j\lle p}R^{(i,j)}(z_i-z_j).
\]
Define a bilinear form $B_{\bla,\bm z}$ on $L(\bla,\bm z)$ by
\[
B_{\bla,\bm z}(w_1,w_2)=B_{\bla}(w_1,R_{\bla,\bm z}w_2),
\]
for all $w_1,w_2\in L(\bla,\bm z)$.

One shows that, c.f. \cite[Section 7]{MTV},
\[
B_{\bla,\bm z}(|0\rangle,|0\rangle)=1,\qquad B_{\bla,\bm z}(Xw_1,w_2)=(-1)^{|X||w_1|}B_{\bla,\bm z}(w_1,\iota(X)w_2), 
\]
for all $X\in \rY(\gl_{1|1})$, $w_1,w_2\in L(\bla,\bm z)$. In addition, if $L(\bla,\bm z)$ is irreducible, then $B_{\bla,\bm z}$ is non-degenerate.

\subsection{Bethe ansatz for $\gl_{1|1}$ XXX model}\label{sec 7.2}
In this section, we study the spectrum of the transfer matrix $\cT(x)=\mc L_{11}(x)-\mc L_{22}(x)$. 

Let $\bla=(\la^{(1)},\dots,\la^{(p)})$ be a sequence of non-degenerate $\gl_{1|1}$ weights. Recall from Section \ref{sec repro pro gl11} that
if $y=(x-t_1)\cdots(x-t_l)$ is a divisor of $\varphi(x)-\psi(x)$, then $\bm t=(t_1,\dots,t_l)$ is a solution of the BAE associated to $\bm s_0$, $\bm\lambda$, $\bm z$, and $l$. 

It is convenient to renormalize the Bethe vector $w(\bs t,\bs z)$ associated to $\bm t$, see \eqref{BV 2},:
\[
\tilde w(\bs t,\bs z)= c_0 w(\bs t,\bs z), \qquad c_0=\prod_{i=1}^l \prod_{k=1}^p(t_i-z_k).
\]
The factor $c_0$ clears up the denominators and the Bethe vector $\tilde w(\bs t,\bs z)$ is well-defined for all $\bs z,\bs t$.

The following theorem is well known, see e.g. \cite{BR}.
\begin{thm}\label{thm gl11 eigenvalue in y}
	 If the Bethe vector $\tilde w(\bs t,\bs z)$ is non-zero, then $\tilde w(\bs t,\bs z)$ is an eigenvector of the transfer matrix $\cT(x)$ with the corresponding eigenvalue 
	\beq\label{eq gl11 eigenvalue in y}
	\mc E(x)=\frac{y[1]}{y}(\varphi-\psi)\prod_{k=1 }^p(x-z_k)^{-1}.
	\eeq
\end{thm}
\begin{proof}
	For $j=1,2$, one has the following relation,
	\begin{align}
	\mc L_{jj}(x)\mc L_{12}(t_1)\cdots \mc L_{12}(t_l)=&\ \xi(x;\bm t)\mc L_{12}(t_1)\cdots \mc L_{12}(t_l)\mc L_{jj}(x)\nonumber \\
	&\ +\sum_{i=1}^{l} \xi_i(x;\bm t)\mc L_{12}(x)\mc L_{12}(t_1)\cdots\widehat{\mc L_{12}(t_i)}\cdots \mc L_{12}(t_l)\mc L_{jj}(t_i).\label{eq com A D B}
	\end{align}
	Here the symbol $\widehat{\mc L_{12}(t_i)}$ means the factor $\mc L_{12}(t_i)$ is skipped and the functions $\xi(x;\bm t)$ and $\xi_i(x;\bm t)$ are given by
	\[
	\xi(x;\bm t)=\prod_{1\lle r\lle l}\frac{x-t_r-1}{x-t_r}=\frac{y[1]}{y},\quad
	\xi_i(x;\bm t)=(-1)^{i-1}\frac{1}{x-t_i}\prod_{1\lle r<i}\frac{t_i-t_r+1}{t_i-t_r}\prod_{ i<r\lle l}\frac{t_i-t_r-1}{t_i-t_r}.
	\]
	
	We have
	\begin{align*}
	\cT(x)|0\rangle=(\varphi-\psi)\prod_{k=1 }^p(x-z_k)^{-1}|0\rangle.
	\end{align*}
	Since $\bm t$ is a solution of the BAE, we have $c_0 \cT(t_i)|0\rangle=0$ for $i=1,\dots,l$. Therefore it follows from \eqref{eq com A D B} that
	\[
	\cT(x)\tilde w(\bs z,\bs t)=c_0(\mc L_{11}(x)-\mc L_{22}(x))\mc L_{12}(t_1)\cdots \mc L_{12}(t_l)|0\rangle=\frac{y[1]}{y}(\varphi-\psi)\prod_{k=1 }^p(x-z_k)^{-1}\tilde w(\bs z,\bs t).\qedhere
	\]
\end{proof}

Recall that the transfer matrix $\cT(x)$ commutes with the subalgebra $\mathrm U(\gl_{1|1})$ of $\rY(\gl_{1|1})$. 

\begin{prop}\label{prop BV gl11 singular}
	The Bethe vector $\tilde w(\bs t,\bs z)$ is $\gl_{1|1}$ singular. 
\end{prop}
\begin{proof}
	By \eqref{zero mode com relations}, one has the following relation,
	\[
	[\mc L_{21}^{(1)},\mc L_{12}(t_1)\cdots \mc L_{12}(t_l)]=\sum_{i=1}^{l} \nu_i(\bm t)\mc L_{12}(t_1)\cdots\widehat{\mc L_{12}(t_i)}\cdots \mc L_{12}(t_l)\cT(t_i).
	\]
	The functions $\nu_k(\bm t)$ are given by
	\[
	\nu_i(\bm t)=(-1)^i\prod_{1\lle r<i}\frac{t_i-t_r+1}{t_i-t_r}\prod_{ i<r\lle l}\frac{t_i-t_r-1}{t_i-t_r}.
	\]
	Note that $\mc L_{21}^{(1)}|0\rangle=0$ and $c_0\cT(t_i)|0\rangle=0$ for $i=1,\dots,l$,
	therefore the statement  follows.
\end{proof}

\begin{prop}\label{prop BV orthogonal}
	Suppose $\varphi\ne \psi$. Let $\bm t$ and $\tl{\bm t}$ be two different solutions of Bethe ansatz equation associated to $\bm s_0$, $\bm\lambda$, $\bm z$, then the Bethe vectors $\tilde w(\bm t,\bm z)$ and $\tilde w(\tl{\bm t}, \bm z)$ are orthogonal with respect to the form $B_{\bla,\bm z}$.
\end{prop}
\begin{proof}
	Let $y$ and $\tl y$ represent $\bm t$ and $\tl{\bm t}$ respectively. Note that we have
	\[
	B_{\bla,\bm z}(\cT(x)\tl w(\bm t,\bm z),\tl w(\tl{\bm t},\bm z))=B_{\bla,\bm z}(\tl w(\bm t,\bm z),\cT(x)\tl w(\tl{\bm t},\bm z)).
	\]
	It follows from Theorem \ref{thm gl11 eigenvalue in y} that
	\[
	\Big(\frac{y[1]}{y}-\frac{\tl y[1]}{\tl y}\Big)(\varphi-\psi)\prod_{k=1 }^p(x-z_k)^{-1}B_{\bla,\bm z}(\tilde w(\bm t,\bm z),\tilde w(\tl{\bm t},\bm z))=0.
	\]
	Since $y$ and $\tl y$ are linearly independent and $\varphi\ne \psi$, the statement follows. 
\end{proof}

The following theorem is a particular case of \cite[Theorem 4.1]{PRS} which asserts that the square of the norm of the Bethe vector is essentially given by the Jacobian of the BAE. 
\begin{thm}[\cite{PRS}]\label{thm norm BV}
	The square of the norm of the Bethe vector $\tilde w(\bm t,\bm z)$ is given by
	\begin{align*}
	B_{\bla,\bm z}(\tilde w(\bm t,\bm z),\tilde w(\bm t,\bm z))=&\ (-1)^{l(l-1)/2}\prod_{1\lle i<j\lle l}\Big(\frac{t_i-t_j-1}{t_i-t_j}  \Big)^2\\ \times\prod_{i=1}^l\prod_{k=1}^p& \big((t_i-z_k+a_k)(t_i-z_k-b_k)\big)\prod_{i=1}^{l}\Big(\sum_{k=1}^{p}\frac{a_k+b_k}{(t_i-z_k+a_k)(t_i-z_k-b_k)}\Big).\qedd
	\end{align*}
\end{thm}

\begin{thm}\label{thm complete generic}
	Suppose $a+b\ne 0$. For generic $\bm z$, the Bethe ansatz is complete. In other words, there are exactly $2^{p-1}$ solutions $\bm t_i$, $i=1,\dots,2^{p-1}$, to the BAE associated to $\bm s_0$, $\bm\lambda$, $\bm z$, and $l$ such that the corresponding Bethe vectors $\tilde w(\bm t_i,\bm z)$, $i=1,\dots,2^{p-1}$, form a basis of $L(\bla,\bm z)^{\rm sing}$.
\end{thm}
\begin{proof}
	Since $a+b\ne0$, we have $\deg(\varphi-\psi)=p-1$. It is not difficult to see that $\dim L(\bla)^{\rm sing}=2^{p-1}$ and for generic $\bm z$ there are exactly $2^{p-1}$ distinct monic divisors of the polynomial $\varphi-\psi$. Each monic divisor of $\varphi-\psi$ corresponds to a solution $\bm t_i$, $i=1,\dots,2^{p-1}$, of BAE associated to $\bm s_0$, $\bm\lambda$, $\bm z$, with possibly different $l$. Due to Proposition \ref{prop BV gl11 singular} and Theorem \ref{thm norm BV}, the Bethe vectors $\tilde w(\bm t_i,\bm z)$ are singular and non-zero. Moreover, it follows from Proposition \ref{prop BV orthogonal} that $\tilde w(\bm t_i,\bm z)$, $i=1,\dots,2^{p-1}$, are linearly independent and hence form a basis of $L(\bla,\bm z)^{\rm sing}$.
\end{proof}

Let $\la^{(k)}=(1,0)$ and $z_k=0$ for all $k=1,\dots,p$. This case is the \emph{homogeneous super} XXX \emph{model}. We obtain the completeness of homogeneous super XXX model.

Let $\theta$ be a primitive $p$-th root of unity. Set $\vartheta_i=1/(\theta^i-1)$, $i=1,\dots,p-1$.

\begin{cor}
	The Bethe ansatz is complete for super homogeneous XXX model. Explicitly, the Bethe vectors form a basis of $\big((\C^{1|1})^{\otimes p}\big)^{\rm sing}$ and the transfer matrix $\cT(x)$ acts on $\big((\C^{1|1}(0))^{\otimes p}\big)^{\rm sing}$ diagonally with simple spectrum. Moreover, the spectrum of $\cT(x)$ acting on $\big((\C^{1|1}(0))^{\otimes p}\big)^{\rm sing}$ is given by
	\[
	\left\{\frac{(x-\vartheta_{i_1}-1)\cdots(x-\vartheta_{i_l}-1)}{(x-\vartheta_{i_1})\cdots(x-\vartheta_{i_l})}\cdot\frac{(x+1)^p-x^p}{x^p},~ 1\lle i_1<i_2<\dots<i_l\lle p-1, l=0,\dots,p-1\right\}.
	\]
\end{cor}
\begin{proof}
	Note that $\varphi(x)=(x+1)^p$ and $\psi(x)=x^p$. Clearly, we have $\varphi-\psi=p(x-\vartheta_1)\cdots(x-\vartheta_{p-1})$. It is easy to see that $\vartheta_i-\vartheta_j\ne 0,1$ for $i\ne j$ and $\vartheta_i\notin \Z$. Therefore we have exactly $2^{p-1}$ distinct monic divisors 
	\[
	(x-\vartheta_{i_1})\cdots(x-\vartheta_{i_l}),~ 1\lle i_1<i_2<\dots<i_l\lle p-1,~l=0,\dots,p-1,
	\]of the polynomial $\varphi-\psi$ and hence $2^{p-1}$ different solutions $\bm t_i$, $i=1,\dots,2^{p-1}$, of BAE. Therefore, as in Theorem \ref{thm complete generic}, the Bethe wectors $\tilde w(\bm t_i,\bm z)$, $i=1,\dots,2^{p-1}$, form a basis of $\big((\C^{1|1}(0))^{\otimes p}\big)^{\rm sing}$.
\end{proof}

\end{document}